\documentclass[final]{siamltex}

\usepackage{amsmath}
\usepackage{amssymb}
\usepackage{amsfonts}
\usepackage{graphicx}
\def\div{\operatorname{div}}

\def\curl{\operatorname{curl}}

\def\claim#1{\begin{trivlist}\item[\hskip\labelsep\bf#1]\it}
\def\endclaim{\end{trivlist}}

\numberwithin{equation}{section}

\newtheorem{Case}{Case}
\newtheorem{SCase}{Case}

 \newtheorem{remark}[theorem]{Remark}

\title{Reconstruction of interfaces from the elastic farfield measurements using CGO solutions}

\author{
Manas Kar\thanks{RICAM, Austrian Academy of Sciences,
Altenbergerstrasse 69, A-4040, Linz, Austria.
(Email:manas.kar@oeaw.ac.at)
\newline
Supported by the Austrian Science Fund (FWF): P22341-N18.} 
\and { Mourad Sini}
\thanks{RICAM, Austrian Academy of Sciences,
Altenbergerstrasse 69, A-4040, Linz, Austria.
(Email:mourad.sini@oeaw.ac.at)
\newline
Partially supported by the Austrian Science Fund (FWF): P22341-N18.}
}

\begin{document}

\maketitle

\begin{abstract}
In this work, we are concerned with the inverse scattering by interfaces for the linearized and isotropic elastic model at a fixed frequency.
First, we derive complex geometrical optic solutions with linear or spherical phases having a computable dominant part and an 
$H^\alpha$-decaying remainder term with $\alpha <3$, where $H^{\alpha}$ is the classical Sobolev space. Second, based on these properties, we estimate the convex hull as well as non convex parts of the interface using 
the farfields of only one of the two reflected body waves (pressure waves or shear waves) as measurements. The results are given for both the impenetrable obstacles, with traction boundary conditions, and the penetrable
obstacles. In the analysis, we require the surfaces of the obstacles to be Lipschitz regular and, for the penetrable obstacles, the Lam\'e coefficients to be measurable and bounded with the usual jump conditions across
the interface.
\end{abstract}



\pagestyle{myheadings}
\thispagestyle{plain}
\markboth{Reconstruction of interfaces from the elastic farfield}{M. Kar and M. Sini}

\section{Introduction and statement of the result}

\label{sec1} Let $D$ be a bounded and open set of $\mathbb{R}^3$ such that $%
\mathbb{R}^3\setminus \overline{D}$ is connected. The boundary $\partial D$
of $D$ is Lipschitz. We denote by $\lambda$ and $\mu$ the Lam\'e
coefficients and $\kappa$ the frequency. We assume that those
coefficients are measurable, bounded and satisfy the conditions $\mu>0$, $2\mu +3\lambda>0$ and $\mu = \mu_0,$ $\lambda = \lambda_0$ for $x\in \mathbb{R}^3\setminus\overline{D}$ 
with $\mu_0$ and $\lambda_0$ being constants. 
In addition, we set $\lambda_D :=\lambda - \lambda_0$ and $\mu_D := \mu-\mu_0$ and assume that $2\mu_D+3\lambda_D\geq 0$ and $\mu_D>0.$\footnote{ The assumptions 
on the jumps can be relaxed. It is needed only at the vicinity of the points on the interface $\partial D$. In addition, we can also consider the case $2\mu_D+3\lambda_D\leq 0$ and $\mu_D<0$, see Remark \ref{rem}.}

The direct scattering problem can be formulated as follows. Let $u^i$ be an incident field, i,e. a vector field
satisfying $\mu_0 \Delta u^i +(\lambda_0 +\mu_0)\nabla \div u^i+ \kappa^2 u^i=0$ in $\mathbb{R}^3$ and $u^s(u^i)$ be the scattered field associated to the incident field $u^i$. 
In the impenetrable case, the scattering problem reads as follows
\begin{equation}  \label{Lame}
\begin{cases}
\mu_0 \Delta u^s +(\lambda_0 +\mu_0)\nabla \div u^s +\kappa^2 u^s =0, \; \mbox{ in }
\mathbb{R}^3\setminus \overline{D}  \\
\sigma(u^s)\cdot \nu=-\sigma(u^i)\cdot \nu,\; \mbox{ on } \partial D  \\
\lim_{\vert x \vert \rightarrow \infty}\vert x \vert (\frac{\partial u_{p}^{s}}{%
\partial {\vert x \vert}}-i\kappa_p u_{p}^{s})=0 \mbox{ and } \lim_{\vert x \vert
\rightarrow \infty}\vert x \vert (\frac{\partial u_{s}^{s}}{\partial {\vert x \vert%
}}-i\kappa_s u_{s}^{s})=0,
\end{cases}
\end{equation}
where the last two limits are uniform in all the directions $\hat{x}:=\frac{x}{\vert x \vert} \in
\mathbb{S}^2$ with $\sigma(u^s)\cdot \nu:=(2\mu \partial \nu+\lambda \nu \div
+\mu \nu \times \curl)u^s$ and the unit normal vector $\nu$ is directed into
the exterior of $D$.
 In the penetrable obstacle case, the total field $u^t :=u^s + u^i$ satisfies
\begin{equation}  \label{Lame_Penetra}
\begin{cases}
\nabla\cdot(\sigma(u^t)) + \kappa^2 u^t = 0, \; \mbox{ in } \mathbb{R}^3  \\
\lim_{\vert x \vert \rightarrow \infty}\vert x \vert (\frac{\partial u_{p}^{s}}{%
\partial {\vert x \vert}}-i\kappa_p u_{p}^{s})=0 \mbox{ and } \lim_{\vert x \vert
\rightarrow \infty}\vert x \vert (\frac{\partial u_{s}^{s}}{\partial {\vert x \vert%
}}-i\kappa_s u_{s}^{s})=0.
\end{cases}
\end{equation}
The two limits in \eqref{Lame} and \eqref{Lame_Penetra} are called the Kupradze radiation conditions.
For any displacement field $v$, taken as a column vector, the corresponding stress tensor $\sigma(v)$ can be represented as a $3\times 3$ matrix:
$
 \sigma(v) = \lambda(\nabla\cdot v)I_3 + 2\mu\epsilon(v),
$
where $I_3$ is the $3\times 3$ identity matrix and $\epsilon(v) = \frac{1}{2}(\nabla v + (\nabla v)^T)$ denotes the infinitesimal strain tensor. Note that for $v = (v_1,v_2,v_3)^T$, $\nabla v$ denotes the $3\times 3$ matrix whose
$j$-th row is $\nabla v_j$ for $j= 1,2,3.$ Also for a $3\times 3$ matrix function $A$, $\nabla\cdot A$ denotes the column vector whose $j$-th component is the divergence of the $j$-th row of $A$ for $j=1,2,3.$ \\
In both \eqref{Lame} and \eqref{Lame_Penetra}, we denoted $u_{p}^{s}:=-\kappa^{-2}_p \nabla
\div\; u^s$ to be the longitudinal (or the pressure) part of the field $u^s$ and
$u_{s}^{s}:=\kappa^{-2}_s \mathrm{curl \;} \mathrm{curl \;}\; u^s$ to be the
transversal (or the shear) part of the field $u^s$. The constants $\kappa_p:=%
\frac{\kappa}{\sqrt{2\mu_0+\lambda_0}}$ and $\kappa_s:=\frac{\kappa}{\sqrt{\mu_0}}$
are known as the longitudinal and the transversal wave numbers respectively. From the first equality of \eqref{Lame} (and equally the one of \eqref{Lame_Penetra} for $x\in\mathbb{R}^3\setminus\overline{D}$),
 we obtain the well known decomposition of the scattered field $u^s$ as the sum of its
longitudinal and transversal parts, i.e. $u^s=u_{p}^{s}+u_{s}^{s}$. It is well known that
the scattering problems (\ref{Lame}) and \eqref{Lame_Penetra} are well posed, see for instance (\cite%
{H-S},\cite{Kup1}, \cite{Kup2}).

The scattered field $u^s$ has 
the following asymptotic expansion at infinity:
\begin{equation}
u^s(x):=\frac{e^{i\kappa_p \vert x\vert}}{\vert x \vert} u^\infty_p(\hat{x})+%
\frac{e^{i\kappa_s \vert x\vert}}{\vert x \vert} u^\infty_s(\hat{x})+ O(%
\frac{1}{\vert x \vert^2}), \; \vert x \vert \rightarrow \infty
\end{equation}
uniformly in all the directions $\hat{x}\in \mathbb{S}^2$, see
\cite{A-K} for instance.
The fields $u^\infty_p(\hat{x})$ and $u^\infty_s(\hat{x})$ defined on $%
\mathbb{S}^2$ are called correspondingly the longitudinal and transversal
parts of the far field pattern. The longitudinal part $u^\infty_p(\hat{x})$
is normal to $\mathbb{S}^2$ while the transversal part $u^\infty_s(\hat{x})$
is tangential to $\mathbb{S}^2$. It is well known that scattering problems
in linear elasticity occur when we excite one of the two types of incident
plane waves, pressure (or longitudinal) waves and shear (or transversal)
waves. They have the analytic forms $u^p_i(x, d):=d e^{i\kappa_p d\cdot x}$
and $u^s_i(x, d):=d^{\perp}e^{i\kappa_s d\cdot x}$ respectively, where $d^{\perp}$
is any vector in $\mathbb{S}^2$ orthogonal to $d$. Remark that $%
u^p_i(\cdot, d)$ is normal to $\mathbb{S}^2$ and $u^s_i(\cdot, d)$ is
tangential to $\mathbb{S}^2$. As in earlier studies of problems
in elasticity we use superpositions of incident pressure and shear waves
given by
\begin{equation}  \label{plane-incident-waves}
u_i(x, d):=\alpha d e^{i\kappa_p d\cdot x}+\beta d^{\perp}e^{i\kappa_s
d\cdot x}\;
\end{equation}
where $\alpha, \beta \in \mathbb{C}$, $d\in \mathbb{S}^2$.

We denote by $u^{\infty}(\cdot, d, \alpha, \beta)$ the far field pattern
associated with the incident waves of the form (\ref{plane-incident-waves}).
We also denote by $u_p^{\infty}(\cdot, d, \alpha, \beta)$ and $%
u_s^{\infty}(\cdot, d, \alpha, \beta)$ the corresponding longitudinal and
transversal parts of the far field $u^{\infty}(\cdot, d, \alpha, \beta)$,
i.e.
$
u^{\infty}(\cdot, d, \alpha, \beta)=(u_p^{\infty}(\cdot, d, \alpha, \beta),
u_s^{\infty}(\cdot, d, \alpha, \beta)).
$
Hence, we obtain the far field measurements
\begin{equation}  \label{Farfield-matrix}
 (u^p_i, u^s_i)\mapsto F(u^p_i, u^s_i):=\left[%
\begin{array}{ccc}
u_p^{\infty, p}(\cdot, d) & u_p^{\infty, s}(\cdot, d)   \\
u_s^{\infty, p}(\cdot, d) & u_s^{\infty, s}(\cdot, d)
\end{array}%
\right]
\end{equation}
where:

1. $(u_p^{\infty, p}(\cdot, d), u_s^{\infty, p}(\cdot, d))$ is the far field
pattern associated with the pressure incident field $u^p_i(\cdot, d)$.

2. $(u_p^{\infty, s}(\cdot, d), u_s^{\infty, s}(\cdot, d))$ is the far field
pattern associated with the shear incident field $u^s_i(\cdot, d)$. 
\par 
Our concern now is to investigate the following \textbf{\texttt{geometrical
inverse problem:}}

\textit{{From the knowledge of $u^{\infty}(\cdot, d, \alpha, \beta)$ for all
directions $\hat{x}$ and $d$ in $\mathbb{S}^2$ and a couple $(\alpha,\beta)\neq (0, 0)$ in $\mathbb{C}$, determine $D$.}} 
\newline
We can also restate this problem as follows: \textit{{From the knowledge of
the matrix (\ref{Farfield-matrix}) for all directions $\hat{x}$ and $d$ in $%
\mathbb{S}^2$ determine $D$.}}

The first uniqueness result was proved by Hahner and Hsiao, for the model \eqref{Lame}, see \cite{H-S}. It says that every \textit{column} of the
matrix (\ref{Farfield-matrix}) for all directions $\hat{x}$ and $d$ in $%
\mathbb{S}^2$, determines $D$. Later Alves and
Kress \cite{A-K}, Arens \cite{A}, A. Charalambopoulos, D. Gintides
and K. Kiriaki \cite{C-G-K1, C-G-K2, G-K} proposed sampling types methods to solve the inverse problem using the full matrix
(\ref{Farfield-matrix}) for all directions $\hat{x}$ and $d$ in $\mathbb{S}
^2 $. We also mention the works by Guzina and his collaborators using the full near fields \cite{B-G-C-M, G-M, N-G}. We remark that one not only needs the information over all directions of
incidence and measurements, but also both pressure and shear far fields
are necessary. In recent works we proved that it is possible to reduce
the amount of data for detecting $D$ as follows: \\ 
\begin{it}
The knowledge of $u_p^{\infty}(\hat{x}, d, \alpha,
\beta)$ (or respectively $u_s^{\infty}(\hat{x}, d, \alpha, \beta)$) for all
directions $\hat{x}$ and $d$ in $\mathbb{S}^2$ and the couple $(\alpha,
\beta)=(1, 0)$ or $(\alpha, \beta)=(0, 1)$ uniquely determines the obstacle $
D $.
\end{it} \\
In other words, this result says that every \textit{component}
of the matrix (\ref{Farfield-matrix}) for all directions $\hat{x}$ and $d$
in $\mathbb{S}^2$, determines $D$.  In \cite{G-S}, we assumed a $C^4$-regularity of $\partial D$ to prove this result for the
impenetrable obstacle with free boundary conditions, the model \eqref{Lame}.  This $C^4$-regularity is used to derive explicitly the first order term in the asymptotic of the
indicator functions of the probe or singular sources method in terms of the source points.
This regularity is reduced to Lipschitz in \cite{K-S} for both the models \eqref{Lame} and \eqref{Lame_Penetra}, see also \cite{H-K-S} for the rigid obstacles. \\
It is known that these probe/singular sources methods are based on the use of approximating domains isolating the source point of the used point sources, see \cite{Potthast}. 
Since this source point has to move around and near the interface, it creates extra instabilities of the method. 
To overcome this difficulty, Ikehata \cite{Ikeha} proposed the enclosure method  which has the same principle as the probe/singular sources 
methods but instead of using point sources, he uses complex geometric optics type solutions with linear phases, CGOs in short. 
The price to pay is that, in contrast to the point sources with which we obtain the whole interface $\partial D$, we can only obtain the 
intersection of the level-curves of the CGOs with $\partial D$. In the case of linear phases, we can estimate the convex hull, see \cite{Ikeha2} 
for an overview of this method. Later, in the work \cite{K-S-U} by Kenig, Sj{\"o}strand and Uhlmann, other CGOs have been proposed to solve the EIT problem using the localized Dirichlet-Neumann map.
These CGOs have a phase of quadratic form, i.e. behaving as spherical waves. Inspired from these CGOs, Nakamura and Yoshida proposed in \cite{N-Y} 
an enclosure method based on CGOs with spherical waves with which they could estimate, in addition to the convex hull of $D$, some non convex parts of $\partial D$. 
Another family of CGOs is proposed by Uhlmann and Wang \cite{U-W-2} where the phases are the harmonic polynomials in the 2D case. With these CGO, one can recover all the visible 
part, by straight rays, of the interface $\partial D$. We refer the reader to \cite{D-K-S-U} for a classification of these CGOs. 
Regarding the Lam\'e model, the CGOs with linear or spherical phases are constructed in \cite{U-W}, for the stationary case, i.e. $k=0$, while in \cite{Ku} the ones with harmonic 
polynomial phases in the 2D case are investigated. The general form of these CGO's is
\begin{equation}\label{general-form}
 u:=e^{-\tau (\phi+i\psi)}(a +r)
\end{equation}
where $\phi$ is the known and explicit phase we were talking about, $\psi$ is also explicitly known, 
$a$ is either known explicitly or computable and the remainder $r$ is small in the classical Sobolev $H^{\alpha}-$norm, $\alpha <1$, in term of the parameter $\tau$, 
i.e. precisely one has the estimate $\Vert r\Vert_{H^{\alpha}}=O(\tau^{\alpha-1})$, $\alpha \leq 1$.

In our work, we use $p$-parts or $s$-parts of the farfield measurements. The CGOs of the form (\ref{general-form}) are not enough, in particular when 
we use mixed measurements, i.e. $p$-incident (respectively, $s$-incident) waves and $s$-parts (respectively, $p$-parts) of the corresponding farfield patterns.
Instead, we construct CGOs of the form  
\begin{equation}\label{general-form1}
 u:=e^{-\tau (\phi+i\psi)}(a_0+\tau^{-1}a_1 +a_2\tau^{-2} +r),
\end{equation}
with linear or logarithmic phases,
where now $a_0, a_1$ and $a_2$ are either known explicitly or computable and the remainder $r$ is small with the $H^{\alpha}-$norm, $\alpha <3$, 
in term of the parameter $\tau$, i.e. $\Vert r\Vert_{H^{\alpha}}=O(\tau^{\alpha-3})$, $\alpha \leq 3$. Having these CGOs at hand, we state  
the indicator function of the  enclosure method directly from the farfield measurements and use only one of the two body waves (pressure or shear waves). 
Then we justify the enclosure method with no geometrical assumptions on the interface $\partial D$. The analysis is based
on the use of integral equation methods on Sobolev spaces $H^s(\partial D), s \in \mathbb{R}$, for the impenetrable case and $L^p$ estimates of 
the gradients of the solutions of the Lam\'e system with discontinuous Lam\'e coefficients for the penetrable case. This is a generalization to the Lam{\'e} system of the previous works 
\cite{K-S-2} and \cite{S-Y} concerning the Maxwell and acoustic cases respectively.

The paper is organized as follows. In section 2, we define the indicator functions via the farfield pattern. Then using the denseness property of the Herglotz wave functions and the well posedness of the forward problem, 
we link the far field measurements with the CGO solutions. In section 3, the construction of the CGOs is 
discussed while in section 4, we state the main theorem and describe the reconstruction scheme.
In section 5, we prove the main theorem and postpone to section 6 and the appendix, the justification of the needed estimates for the CGOs. 

\section{The indicator functions linking the used farfield parts to the CGOs}
In this section, we follow the procedure of \cite{K-S-1} where we showed the link between farfield and CGOs in the scalar Helmholtz case. So, we start with the following
identity, see for instance Lemma 3.1 in \cite{A-K}:
\begin{equation}\label{Farfield-Nearfield}
\int_{\partial D} \left( U\cdot \overline{\sigma(w_h)\cdot \nu} -\overline{%
w_h} \cdot \sigma(U) \cdot \nu \right) ds(x)=4\pi\int_{\mathbb{S}%
^2}\left(U^{\infty}_p\overline{h_p}+U^{\infty}_s\overline{h_s} \right)ds(\hat x)
\end{equation}
for all radiating fields $U$ and $w_h$ where $w_h$ is the scattered field
associated with the Herglotz field $v_h$ for $h=(h_p, h_s) \in L^2_p(\mathbb{S}%
^2)\times L^2_s(\mathbb{S}^2)$, i.e. $v_h(x):=\int_{\mathbb{S}^2}[e^{i\kappa_p x
\cdot d}h_p(d) +e^{i\kappa_s x \cdot d}h_s(d)]ds(d)$
with $L_{p}^{2}(\mathbb{S}^2) := \{U\in (L^2(\mathbb{S}^2))^3; U(d)\times d=0\}$ while $L_{s}^{2}(\mathbb{S}^2) := \{U\in (L^2(\mathbb{S}^2))^3; U(d)\cdot d=0\}.$\\
Let $v$ be a CGO solution for the Lam{\'e} system and state $\Omega :=\Omega_{cgo}$ to be its domain of definition with $D\subset\subset\Omega$. Examples of these
CGOs will be described in Section 3.
We take its $p$-part $v_p$ and its $s$-part $v_s$. We can find sequences of densities $(h^n_p)_n$ and $(h^n_s)_n$ such that the Herglotz waves
$v_{h^n_p}$ and $v_{h^n_s}$ converge to $v_p$ and $v_s$ respectively on any domain $\tilde\Omega$ containing $D$ and contained in $\Omega$.
These sequences can be obtained as follows. \\
We define $H : (L^2(\mathbb{S}^2))^3 \rightarrow (L^2(\partial\Omega))^3$ as 
$(Hg)(x) := v_g(x).$ We know that $H$ is injective and has a dense range if $\kappa^2$ is not an eigenvalue of the Dirichlet-Lam{\'e} operator on $\Omega$. Due to the monotonicity of these eigenvalue 
in terms of the domains, we change, if needed, $\Omega$ slightly so that $\kappa^2$ is not an eigenvalue anymore. Hence, we can find a sequence $g_n\in(L^2(\mathbb{S}^2))^3$ such that
$Hg_n \rightarrow v$ in $(L^2(\partial\Omega))^3$. Recall that both $Hg_n$ and $v$ satisfy the interior Lam{\'e} problem. By the well-posedness of the interior problem and the interior estimates,
we deduce that $Hg_n \rightarrow v$ in $C^{\infty}(\tilde\Omega)$, since $D\subset\subset\tilde\Omega\subset\subset\Omega.$
Hence, $-\kappa_{p}^{-2}\nabla\nabla\cdot Hg_n \rightarrow v_p$ and $\kappa_{s}^{-2}\curl\curl Hg_n \rightarrow v_s$ in $C^{\infty}(\tilde\Omega).$
But $-\kappa_{p}^{-2}\nabla\nabla\cdot Hg_n = Hh_{p}^{n}$ and $\kappa_{s}^{-2}\curl\curl Hg_n = Hh_{s}^{n}$, where $h_{p}^{n} := d(d\cdot g_n)$ and $h_{s}^{n} := -d\wedge(d\wedge g_n)$. \\
We set $u^s(v_s)$ to be the scattered field corresponding to the $s$ incident wave $v_s$. The $p$ and $s$ parts of the scattered field $u^s(v_s)$ are $u_{p}^{s}(v_s)$ and $u_{s}^{s}(v_s)$ respectively. 
Similarly, we set $u^s(v_p)$ to be the scattered field corresponding to the $p$ incident wave $v_p$. The $p$ and $s$ parts of the scattered field $u^s(v_p)$ are $u_{p}^{s}(v_p)$ and $u_{s}^{s}(v_p)$ respectively.
\subsection{Using longitudinal waves}
Let $(u_{p}^{\infty,p},u_{s}^{\infty,p})$ be the farfield associated to the incident field $de^{i\kappa_pd\cdot x}.$
By the principle of superposition, the farfield associated to the incident field
$
v_g(x):=\int_{\mathbb{S}^2}de^{i\kappa_p d \cdot x}(d\cdot g(d))ds(d)
$
is given by
\begin{equation*}
u_g^{\infty}(\hat x):=(u_g^{\infty, p}(\hat x), u_g^{\infty,
s}(\hat x))=(\int_{\mathbb{S}^2}u_p^{\infty,p}(\hat x, d)(d\cdot g(d))ds(d), \int_{%
\mathbb{S}^2}u_s^{\infty,p}(\hat x, d)(d\cdot g(d))ds(d))
\end{equation*}
where each component is a vector. Replacing in (\ref{Farfield-Nearfield}), using the fact that $Hh_{p}^{n}$ converges to $v_p$ in $C^{\infty}(\tilde\Omega)$, with $D\subset\subset\tilde\Omega$, the trace theorem and the well-posedness of the scattering problem, we obtain:
\begin{align}\
&4\pi \lim_{m,n\rightarrow \infty}\int_{\mathbb{S}^2}\int_{\mathbb{S}%
^2}[u_p^{\infty,p}(\hat x, d)d\cdot g_n(d)]\cdot [\hat x
(\hat x\cdot \overline{g_m(\hat x)})]ds(\hat x)ds(d)&\hspace{2cm} \nonumber\\
&=\int_{\partial D}[u^s(v_p)\cdot (\overline{\sigma(v_p)}\cdot \nu)-\overline{v_p}\cdot(\sigma(u^s(v_p))\cdot \nu)]ds(x)&\hspace{2cm}
\label{log1}\\ 
\text{and}\hspace{2cm}& 4\pi\lim_{m,n\rightarrow \infty}\int_{\mathbb{S}^2}\int_{\mathbb{S}%
^2}[u^{\infty,p}_s(\hat x, d)d\cdot g_n(d)]\cdot [\hat x
\wedge(\hat x\wedge\overline{g_m(\hat x)})]ds(\hat x)ds(d)&\hspace{2cm} \nonumber\\
&=\int_{\partial D}[u^s(v_p)\cdot (\overline{\sigma(v_s)}\cdot \nu)-\overline{v_s}\cdot(\sigma(u^s(v_p))\cdot \nu)]ds(x).&\hspace{2cm}\label{log2}
\end{align}
\subsection{Using shear incident waves}
Let $i_1$ and $i_2$ be two vectors linearly independent and tangent to $\mathbb{S}^2$. Then $d\wedge(d\wedge i_1)$ and $d\wedge(d\wedge i_2)$ are obviously also tangent to $\mathbb{S}^2$ and in addition
they are linear independent. 
Indeed,
$
 \alpha_1d\wedge(d\wedge i_1) + \alpha_2d\wedge(d\wedge i_2) = 0 
  \Leftrightarrow d\wedge(d\wedge(\alpha_1i_1 + \alpha_2i_2)) = 0 
 \Leftrightarrow\alpha_1 = \alpha_2 = 0.
$
Let now $u_{g}^{\infty}(\hat x, d):= (u_p^{\infty,s}(\hat x, d), u_s^{\infty,s}(\hat x, d))$ be the farfield associated with the incident plane wave $d\wedge(d\wedge i_j)e^{i\kappa_s x\cdot d}, j= 1,2.$
Hence the farfield associated with the Herglotz wave
$
 v_g(x) := \int_{\mathbb{S}^2}d\wedge(d\wedge g)e^{i\kappa_sx\cdot d}ds(d) 
$
 is
\[
 v_{g}^{\infty}(\hat x) := \sum_{j=1}^{2}\left(\int_{\mathbb{S}^2}u_{p}^{\infty,s}(\hat x, d)(i_j\cdot g)ds(d), \int_{\mathbb{S}^2}u_{s}^{\infty,s}(\hat x, d)(i_j\cdot g)ds(d)\right),
\]
since 
$
 d\wedge(d\wedge g) = d\wedge(d\wedge i_1)(i_1\cdot g) + d\wedge(d\wedge i_2)(i_2\cdot g).
$
Hence,
\begin{eqnarray}\label{she1}
 &4\pi\lim_{m,n\rightarrow\infty}\sum_{j=1}^{2}\int_{\mathbb{S}^2}\int_{\mathbb{S}^2}[u_{s}^{\infty,s}(\hat x, d)(i_j\cdot g_n)(d)]\cdot[\hat x\wedge\hat x\wedge \overline{g_m(\hat x)}]ds(\hat x)ds(d)\nonumber\\
& = \int_{\partial D} [u^s(v_s)\cdot (\overline{\sigma(v_s)}\cdot\nu)
-\overline{v_s}\cdot(\sigma(u^s(v_s))\cdot\nu)]ds(x), 
\end{eqnarray}
\begin{eqnarray}\label{she2}
\text{and}\hspace{2cm} &4\pi\lim_{m,n\rightarrow\infty}\sum_{j=1}^{2}\int_{\mathbb{S}^2}\int_{\mathbb{S}^2}[u_{p}^{\infty,s}(\hat x, d)(i_j\cdot g_n)(d)]\cdot[\hat x(\hat x\cdot \overline{g_m(\hat x)})]ds(\hat x)ds(d) \nonumber\\
&= \int_{\partial D}[u^s(v_s)\cdot (\overline{\sigma(v_p)}\cdot\nu)
-\overline{v_p}\cdot(\sigma(u^s(v_s))\cdot\nu)]ds(x).&\hspace{2cm}
\end{eqnarray}
\begin{subsection}{\textbf{The indicator functions}}\label{indi_sub}
We set
\begin{eqnarray}\label{PP}
I_{pp}&:=&4\pi\lim_{m,n\rightarrow \infty}\int_{\mathbb{S}^2}\int_{\mathbb{S}
^2}[u_p^{\infty,p}(\hat x, d)d\cdot g_n(d)]\cdot [\hat x
(\hat x\cdot \overline{g_m(\hat x)})]ds(\hat x)ds(d),
\\ \label{PS}
I_{ps}&:=&4\pi\lim_{m,n\rightarrow \infty}\int_{\mathbb{S}^2}\int_{\mathbb{S}%
^2}[u^{\infty,p}_s(\hat x, d)d\cdot g_n(d)]\cdot [\hat x
\wedge(\hat x\wedge \overline{g_m(\hat x)})]ds(\hat x)ds(d),
 \\ \label{SS}
I_{ss}&:=&4\pi\lim_{m,n\rightarrow\infty}\sum_{j=1}^{2}\int_{\mathbb{S}^2}\int_{\mathbb{S}^2}[u_{s}^{\infty,s}(\hat x, d)(i_j\cdot g_n)(d)]\cdot[\hat x\wedge\hat x\wedge \overline{g_m(\hat x)}]ds(\hat x)ds(d), 
\\ \label{SP} 
\text{and}\hspace{2cm} I_{sp}&:=&4\pi\lim_{m,n\rightarrow\infty}\sum_{j=1}^{2}\int_{\mathbb{S}^2}\int_{\mathbb{S}^2}[u_{p}^{\infty,s}(\hat x, d)(i_j\cdot g_n)(d)]\cdot[\hat x(\hat x\cdot \overline{g_m(\hat x)})]ds(\hat x)ds(d).\hspace{2cm}
\end{eqnarray}
\end{subsection}
Therefore, the indicator function $I_{pp}$ is defined based on $p$-parts of the far field associated to $p$-incident wave. Correspondingly $I_{ps}$ depends on $s$-part of the far field associated to $p$-incident wave, 
$I_{ss}$ depends on $s$-part of the far field associated to the $s$-incident wave and finally $I_{sp}$ depends on $p$-part of the far field associated to the $s$-incident wave.
\begin{section}{Construction of CGO solutions}\label{ppssty}
\begin{subsection}{CGO solutions for $I_{pp}$ and $I_{ss}$}\label{cgo_PP}
Let us assume $u$ to be a solution for the following Lam\'e system
\begin{equation}\label{CGO_1}
  \mu_0\Delta u + (\lambda_0 + \mu_0)\nabla \div u + \kappa^2 u = 0.
\end{equation}
Applying the identity $\curl\curl = \nabla\nabla\cdot - \Delta$ in \eqref{CGO_1}, we have
\[
  u
 = - \frac{\lambda_0 + 2\mu_0}{\kappa^2} \nabla\nabla\cdot u + \frac{\mu_0}{\kappa^2} \curl\curl u 
   =: u_p + u_s 
\]
where $u_p$ and $u_s$ are the $p$-part and $s$-part of the solution $u$ respectively.
\subsubsection{$p$-type CGOs for $I_{pp}$}\label{subpp}
Taking $\nabla\cdot$ in both sides in \eqref{CGO_1}, we obtain
\[
 \mu_0\Delta(\nabla\cdot u) + (\lambda_0 + \mu_0)\Delta(\nabla\cdot u) = - \kappa^2(\nabla\cdot u).
\]
Define $U := \nabla\cdot u$. Therefore $U$ satisfies the Helmholtz equation
$
 \Delta U + \kappa_{p}^{2} U = 0.
$
Hence the $p$-part of $u$ is $-\frac{\lambda_0 + 2\mu_0}{\kappa^2} \nabla U$. This suggest to take the CGO solution for the Lam\'e system 
of the form
\begin{equation}\label{pp}
 \nabla U
\end{equation}
 where
$U$ is the CGO solution for scalar Helmholtz equation, i.e, $U$ satisfies 
\begin{equation}\label{pHel}
 \Delta U + \kappa_{p}^{2} U = 0.
\end{equation}
The resulting vector field $\nabla U$, in \eqref{pp}, is also a solution of \eqref{CGO_1} and it is of $p$-type since its $s$-part is zero. 
\subsubsection{$s$-type CGOs for $I_{ss}$}\label{cgo_SS}
  Define $V =: \curl u$, where $u$ satisfies \eqref{CGO_1}. It satisfies the vector Helmholtz equation
\begin{equation}\label{sHel}
 \Delta V + \kappa_{s}^{2} V = 0.
\end{equation}
The $s$-part of the solution $u$ is $\frac{\mu_0}{\kappa^2}\curl V.$ This suggests to take 
the CGO solution for the Lam\'e system of the form
\begin{equation}\label{ss}
 \curl V.
\end{equation}
The resulting vector field $\curl V$, in \eqref{ss}, is also a solution of \eqref{CGO_1} and it is of $s$-type since its $p$-part is zero.
\subsubsection{The forms of the corresponding CGOs}
In the following proposition, we provide the forms of the CGOs for the Helmholtz type equations in \eqref{pHel} and \eqref{sHel}. The CGOs we use for $I_{pp}$ and $I_{ss}$ are then deduce using \eqref{pp} and \eqref{ss} respectively.
\begin{proposition}\label{ll}
 \begin{enumerate}
  \item 
(Linear phase.)\\
Let $\rho,\rho^{\perp}\in \mathbb{S}^2,$ with $\rho\cdot\rho^{\perp}=0$ and $t,\tau>0.$ The function
$
 U(x;\tau,t) := e^{\tau(x\cdot\rho-t)+i\sqrt{\tau^2+\kappa_{l}^{2}}x\cdot\rho^{\perp}}
$
satisfies $(\Delta + \kappa_{l}^{2})U=0$ in $\mathbb{R}^3,$ for $l=s$ or $l=p.$
\item
(Logarithmic phase.)\\
Let $\Omega$ be a $C^2$-smooth domain and set $ch(\Omega)$ to be its convex hull.
Choose $x_0 \in\mathbb{R}^3\setminus\overline{ch(\Omega)}$ and let $\omega_0\in \mathbb{S}^2$ be a vector such that
$
 \{x\in\mathbb{R}^3 ; x-x_0 = \lambda\omega_0, \lambda\in \mathbb{R}\}\cap \partial\Omega = \emptyset.
$
Then there exists a solution of $(\Delta + \kappa_{l}^{2})U=0$ of the following form
\begin{equation}\label{keti}
U(x;\tau,t) := e^{\tau(t-\log\vert x-x_0\vert)-i\tau\psi(x)}(a_0 + \tau^{-1}a_1 + r)
\end{equation}
with $a_0$ and $a_1$ are smooth and computable functions which depend on $\kappa_l, l=s$ or $l=p,$ and $r$ has the following behavior
\begin{equation}\label{es_r}
  \Vert r\Vert_{H_{scl}^{2}(\Omega)} \leq C\tau^{-2}, \ \text{$C>0$ is a universal constant,}\
\end{equation}
where $H_{scl}^{2}(\Omega)$ is the semi-classical Sobolev space defined by
$
 H_{scl}^{2}(\Omega) := \{V\in L^2(\Omega)/ (\tau^{-1}\partial)^{\alpha}V\in L^2(\Omega), |\alpha|\leq 2\},
$
equipped with the norm
$
\|V\|_{H_{scl}^{2}(\Omega)}^{2} :=\sum_{|\alpha|\leq 2}\|(\tau^{-1}\partial)^{\alpha}V\|_{L^2(\Omega)}^{2}.
$
From \eqref{es_r}, we have in particular
\begin{equation}\label{winkl}
  \Vert r\Vert_{H^{s}(\Omega)} \leq C\tau^{-(2-s)}, \ 0 \leq s \leq 2.
\end{equation}
\end{enumerate}
\end{proposition}
\begin{proof}
 The CGOs with linear phases are given in \cite{Ikeha}. The CGOs with log-phases are given in \cite{K-S-U}
for $\kappa=0$ of the form 
$
 e^{\tau(t-\log\vert x-x_0\vert)-i\tau\psi(x)}(a_0 + r)
$
and 
$
 \|r\|_{H_{scl}^{1}(\Omega)}\leq c\tau^{-1}.
$
The CGOs stated in this proposition, with the corresponding estimate \eqref{es_r} of the remainder term, are given in \cite{S-Y}.
\end{proof}
\end{subsection}
\begin{subsection}{CGO solutions for $I_{sp}$ and $I_{ps}$}\label{cgo_SPPS}
The natural CGOs introduced in section \ref{subpp} are useful for $I_{pp}$. However, with such functions as incident waves the other indicator functions $I_{ps}, I_{ss}$ and $I_{sp}$ 
vanish, see \eqref{log2}, \eqref{she1} and \eqref{she2} respectively, since the $s$-part, $v_s$, of the
CGO solution is nulle, and hence are not useful.  Similarly, the CGOs constructed in section \ref{cgo_SS} are useful for $I_{ss}$ but not for
$I_{pp}$, $I_{sp}$ and $I_{ps}$. To construct CGOs useful for $I_{sp}$ and $I_{ps}$, we need to consider the full Lam{\'e}
 system, i.e. solutions with non vanishing $p$ and $s$ parts. To construct such CGOs, we follow the approach by Uhlmann-Wang \cite{U-W}. 
The non-divergence form of the isotropic elasticity system can be written as
\begin{equation}\label{4.1.3.1}
 \mu_0\Delta u + (\lambda_0 + \mu_0)\nabla(\nabla\cdot u) + \kappa^2 u = 0 \ \ \text{on} \ \Omega.
\end{equation}
Let 
 $ W = \left
     (\begin{array}{c} 
               w \\ 
               g
     \end{array}
      \right)$
   satisfy 
\begin{equation}\label{4.1.3.2}
PW := \Delta \left
     (\begin{array}{c} 
               w \\ 
               g
     \end{array}
      \right) 
 + A \left
     (\begin{array}{c} 
               \nabla g \\ 
               \nabla\cdot w
     \end{array}
      \right)
+ Q \left
     (\begin{array}{c} 
               w \\ 
               g
     \end{array}
      \right) = 0,     
\end{equation} 
where 
   $ A = \left
     (\begin{array}{cc} 
               0 & 0 \\ 
               0 & \frac{\lambda_0 + \mu_0}{\lambda_0 + 2\mu_0} \mu_{0}^{\frac{1}{2}}
     \end{array}
      \right) 
\text{and}\ \
 Q = \kappa^2\left
     (\begin{array}{cc} 
               \frac{1}{\mu_0} & 0 \\ 
               0 & \frac{1}{\lambda_0 + 2\mu_0} 
     \end{array}
      \right). $  
  Then 
$ 
u := \mu_{0}^{-\frac{1}{2}}w + {\mu_0}^{-1}\nabla g
$
 satisfies \eqref{4.1.3.1}. Consider now the matrix operator $P_{\tau^{-1}} = -\tau^{-2}P.$ Then the operator $P$ in \eqref{4.1.3.2} turns out to be the following operator 
\[
 P_{\tau^{-1}} = (\tau^{-1}D)^2 + i\tau^{-1}A_1\left
                                                   (\begin{array}{c} 
                                                          \tau^{-1}D \\ 
                                                          \tau^{-1}D\cdot
                                                    \end{array}
                                                    \right) 
    + \tau^{-2} A_0,\]
where $D := -i\nabla, A_1 := -A $ and $ A_0 := -Q.$ Later on we shall also denote the matrix operator
\[
         i A_1\left
                    (\begin{array}{c} 
                     \tau^{-1}D \\ 
                     \tau^{-1}D\cdot
                     \end{array}
                      \right) = A_1(\tau^{-1}D).
\]
 
            Using semi classical Weyl calculus, the derivation of the Carleman estimate with semiclassical $H^{-2}$ norm for $P_{\tau^{-1}}$ can be found in \cite{U-W}. We state it in the following proposition.
\begin{proposition}\label{4.1.3p1}
       Let $\varphi(x)$ be a linear or logarithmic phase. If $\tau$ is large enough, then for any $F\in L^2(\Omega)$, there exists $V\in H_{scl}^{2}(\Omega)$ such that
\[
 e^{\tau\varphi} P_{\tau^{-1}}(e^{-\tau\varphi}V) = F, \ \ \Vert V\Vert_{H_{scl}^{2}(\Omega)} \leq C \tau \Vert F\Vert_{L^2(\Omega)}
\]
with some constant $C > 0.$ 
\end{proposition}
\begin{remark}
 In \cite{U-W}, the functions $w$ and $g$ are represented in the form 
\[
 w = e^{-\tau(\varphi+i\psi)}(l+r) \ \text{and} \ \  g = e^{-\tau(\varphi+i\psi)}(d+s),
\]
where  $l,r \in {\mathbb{C}}^3$, $l, d$ are smooth and $(r, ~s)^\top$
satisfy the estimate
\[
 \Vert \partial^{\alpha}R\Vert_{L^2(\Omega)} \leq C \tau^{\vert\alpha\vert-1} \ \text{for} \ \vert\alpha\vert\leq 2.
\]
However, this estimate is not enough for our analysis of $I_{sp}$ and $I_{ps}$ cases. Indeed, we need the boundedness of the remainder term in the norm of the Sobolev space $H^3(\Omega)$. In the following proposition the representation of the CGO solutions with stronger estimate is given.
\end{remark}
\begin{proposition}\label{CGO_lin}
\begin{enumerate}
 \item (Linear phase).
     Choose $\rho,\rho^{\bot} \in {\mathbb{S}}^2 :=\{x\in {\mathbb{R}}^3 ; \vert x\vert = 1\}$ with $\rho\cdot\rho^{\bot} = 0.$ Then there exists $w$ and $g$ of the forms
\begin{equation}\label{monro}
 \begin{split}
& w = e^{\tau((x\cdot\rho - t)+ix\cdot\rho^{\bot})}(a_0 + \tau^{-1}a_1 + \tau^{-2}a_2 + r),\\
& g = e^{\tau((x\cdot\rho - t)+ix\cdot\rho^{\bot})}(b_0 + \tau^{-1}b_1 + \tau^{-2}b_2 + s),
\end{split}
\end{equation}
respectively, satisfying \eqref{4.1.3.2}, where $\tau(> 0)$ and $t\in \mathbb{R}$ are parameters, $a_0(x)$ is a smooth non-vanishing $3\times 1$ vector valued complex function on $\overline{\Omega}$ and $ a_1(x), a_2(x), b_0(x), b_1(x), b_2(x)$ are all
smooth functions in $\overline{\Omega}$ with $a_1, a_2 \in {\mathbb{C}}^3$ and $b_0, b_1, b_2 \in \mathbb{C}.$ The coefficients $a_j$ and $b_j, j=0,1,2,$ are all computable. 
\item
(Logarithmic phase.)
Choose $x_0 \in\mathbb{R}^3\setminus\overline{ch(\Omega)}$ and let $\omega_0\in \mathbb{S}^2$ be a vector such that
\begin{equation}\label{mcon}
 \{x\in\mathbb{R}^3 ; x-x_0 = \lambda\omega_0, \lambda\in \mathbb{R}\}\cap \partial\Omega = \emptyset.
\end{equation}
Then there exists $w$ and $g$ of the form \footnote{The parameter $a_j$ and $b_j, j=0,1,2,$ are not necessary the same as in the case of CGOs with linear phases. However, we
keep the same symbols to avoid heavy notations in the sections 6 and 7, in particular. }
\begin{equation}\label{meri}
\begin{split}
& w = e^{\tau(t-\log\vert x-x_0\vert)-i\tau\psi(x)}(a_0 + \tau^{-1}a_1 + \tau^{-2}a_2 + r),\\
& g = e^{\tau(t-\log\vert x-x_0\vert)-i\tau\psi(x)}(b_0 + \tau^{-1}b_1 + \tau^{-2}b_2 + s),
\end{split}
\end{equation}
satisfying \eqref{4.1.3.2} where $\tau(> 0), t\in \mathbb{R}$ are parameters, $a_0(x)$ is a smooth non-vanishing $3\times 1$ vector valued complex function on $\overline{\Omega}$ and $ a_1(x), a_2(x), b_0(x), b_1(x), b_2(x)$ are all
smooth functions on $\overline{\Omega}$ with $a_1, a_2 \in {\mathbb{C}}^3$ and $b_0, b_1, b_2 \in \mathbb{C}.$ 
$\psi(x)$ is defined by
\begin{equation*}
 \psi(x) := d_{\mathbb{S}^2}\left(\frac{x-x_0}{\vert x-x_0\vert}, \omega_0\right)
\end{equation*}
with the distance function $d_{\mathbb{S}^2}(\cdot,\cdot)$ on $\mathbb{S}^2$. The coefficients $a_j, b_j, j=0,1,2,$ are all computable. 
\end{enumerate}
In addition, for both linear and logarithmic phases, the remainder term $R :=(r, ~s)^\top$
 enjoys the estimates: 
\begin{equation}\label{try}
 \Vert \partial^{\alpha}R\Vert_{L^2(\Omega)} \leq C \tau^{\vert\alpha\vert-3} \ \text{for} \ \vert\alpha\vert\leq 2
\end{equation}
and
\begin{equation}\label{korc}
 \Vert \nabla R\Vert_{H^2(\Omega_0)} \leq C \ \text{for any $\Omega_0\subset\subset \Omega$} \ \text{as} \ \tau\rightarrow\infty.
\end{equation}
Finally, for both linear and logarithmic phases
$
u := {\mu_0}^{-\frac{1}{2}}w + {\mu_0}^{-1}\nabla g
$
is a complex geometrical optics solution for \eqref{4.1.3.1}. The p-part and s-part of these CGO solutions, denoted by $u_p$ and $u_s$ respectively, are 
represented by
\begin{eqnarray}\label{lp}
 u_p &=& -\kappa_{p}^{-2}[\mu_{0}^{-\frac{1}{2}}\nabla(\nabla\cdot w) + \mu_{0}^{-1}\nabla(\Delta g)],
\\ \label{ls}
 \text{and}\hspace{2cm} u_s &=& -\kappa_{s}^{-2}\mu_{0}^{-\frac{1}{2}}[\nabla(\nabla\cdot w) - \Delta w].\hspace{2cm}
\end{eqnarray}
\end{proposition}
\begin{proof}
 First, remark that we can get rid of the constant terms $e^{-\tau t}$ and $e^{\tau t}$ in \eqref{monro} and \eqref{meri}. We construct the solution of \eqref{4.1.3.2} of the form
\[
 U = e^{-\tau(\varphi+i\psi)}(L_0 + \tau^{-1} L_1 + \tau^{-2} L_2 + R),
\]
where $(\varphi + i\psi)$ is a phase function, $L_0, L_1, L_2$ are smooth functions in $\overline{\Omega}$ and $R\in H^2(\Omega)$ is the remainder. 
Applying the WKB method for the conjugate operator $e^{\tau\varphi}\tau^{-2}Pe^{-\tau\psi}$, we obtain
\begin{equation}\label{conju_op}
\begin{split}
& e^{\tau\varphi}\tau^{-2}Pe^{-\tau(\varphi+i\psi)}(L_0 + \tau^{-1} L_1 + \tau^{-2} L_2) \\
& = e^{-i\tau\psi}\left[(-(\nabla(\varphi+i\psi))^2 + \tau^{-1}{\tilde Q} + \tau^{-2}P)(L_0 + \tau^{-1} L_1 + \tau^{-2} L_2) \right] \\
& = e^{-i\tau\psi}
[-(\nabla(\varphi+i\psi))^2L_0 
  + \tau^{-1}\{-(\nabla(\varphi+i\psi))^2 L_1 + {\tilde Q}L_0\}          \\   
& + \tau^{-2}\{-(\nabla(\varphi+i\psi))^2 L_2 + {\tilde Q}L_1 + PL_0\} 
 + \tau^{-3}\{{\tilde Q}L_2 + PL_1 \} 
 + \tau^{-4}\{PL_2\}], 
\end{split}
\end{equation}
where $\tilde Q := -\nabla\psi\cdot D - D\cdot\nabla\psi + i\nabla\varphi\cdot D + iD\cdot\nabla\varphi + A_1(i\nabla\varphi-\nabla\psi).$
We choose $\varphi, \psi, L_0, L_1$ and $L_2$ such that
\begin{eqnarray}\label{csw1}
 (\nabla(\varphi+i\psi))^2 = 0,
\\ \label{csw2}
 -(\nabla(\varphi+i\psi))^2 L_1 + {\tilde Q}L_0 = 0,
\\ \label{csw3}
 -(\nabla(\varphi+i\psi))^2 L_2 + {\tilde Q}L_1 + PL_0 = 0,
\\ \label{csw4}
 {\tilde Q}L_2 + PL_1 = 0.
\end{eqnarray}
Therefore, \eqref{csw1} is the well known eikonal equation for $\varphi$ and $\psi$ can be written as
\begin{equation}\label{eikonal_csw}
 (\nabla \psi)^2 = (\nabla \varphi)^2, \\ \nabla\varphi\cdot\nabla\psi = 0. 
\end{equation}
\begin{Case}{Logarithmic Phase.}
\end{Case}
For the case of logarithmic phase, we have $\varphi(x) = \log\vert x-x_0\vert$. So we can find a solution of \eqref{eikonal_csw} of the form
\[
 d_{\mathbb{S}^2}\left( \frac{x-x_0}{\vert x-x_0\vert}, \omega_0\right) = \frac{\pi}{2} - \tan^{-1}\frac{\omega_0\cdot(x-x_0)}{\sqrt{(x-x_0)^2-(\omega_0\cdot(x-x_0))^2}},
\]
see \cite{S-Y} for instance.
If we choose $\psi$ as above, \eqref{csw2}, \eqref{csw3} and \eqref{csw4} are the transport equations for $L_0, L_1$ and $L_2.$ Now by the change of coordinates so that $x_0 = 0, \overline{\Omega} \subset\{x_3 > 0\},$ and $w_0 = e_1,$
we obtain $\varphi + i\psi = \log z,$ where $z = x_1 + i\vert x'\vert$ is a complex variable with $x' := (x_2,x_3).$ Therefore, if we write $L_0 :=(a_0, ~b_0)^\top$
then \eqref{csw2} gives
\begin{eqnarray}\label{trans_1}
\left[(\nabla\log z)\cdot\nabla + \nabla\cdot\nabla\log z \right]a_0 = 0
\\ \label{trans_2}
 \text{and}\hspace{2cm}\left[(\nabla\log z)\cdot\nabla + \nabla\cdot\nabla\log z \right]b_0 + \left(\frac{\lambda_0 + \mu_0}{\lambda_0 + 2\mu_0} \right){\mu_0}^{\frac{1}{2}}(\nabla\log z)\cdot a_0 = 0.&\hspace{2cm}
\end{eqnarray}
The equation \eqref{trans_1} reduces to the following Cauchy-Riemann equation in the $z$ variable
\[
 \left(\partial_{\bar z} - \frac{1}{2(z-\bar z)} \right)a_0(z,\theta) = 0,
\]
where $\theta = \frac{x'}{\vert x'\vert}.$ This last equation has the following solution
\begin{equation}\label{solCR}
 a_0(z,\theta) = (z-\bar z)^{-\frac{1}{2}} (\beta_1, ~\beta_2, ~\beta_3)^\top
\end{equation}
where $(\beta_1, \beta_2, \beta_3) \in\mathbb{S}^2$ are arbitrary.\footnote{In section 6.1, see the proof of Lemma \ref{ne1}, we choose $(\beta_1, \beta_2, \beta_3) = \omega_0$, where $\omega_0$ is given in \eqref{mcon}.}
Replacing $a_0$ in \eqref{trans_2}, we obtain smooth solution $b_0$ in $\Omega.$ Similarly, we get smooth solutions $L_1$ and $L_2$ of the equations
\eqref{csw3} and \eqref{csw4} on $\overline{\Omega}$, see \cite{A-F} for instance.
Finally, to derive the estimate of the remainder term $R$, from Proposition \ref{4.1.3p1}, we have
\begin{eqnarray}\label{csw_es_remainder0}
 e^{\tau\varphi}\tau^{-2} Pe^{-\tau(\varphi + i\psi)}R = -e^{-i\tau\psi}\tau^{-4}(PL_2), \\ \nonumber
 \Vert R\Vert_{H_{scl}^{2}(\Omega)}\leq C \tau\Vert e^{-i\tau\psi}\tau^{-4}(PL_2)\Vert_{L^2(\Omega)} = O(\tau^{-3}).
\end{eqnarray}

Now, to estimate the gradient of $R$ in $H^2(\Omega_0)$ for sub-domain $\Omega_0\subset\subset\Omega,$ we use interior estimate.
The equation \eqref{csw_es_remainder0} can be rewritten in terms of $r, s$ as 
\begin{equation*}\label{csw_es_remainder1}
\begin{split}
-\Delta r + 2\tau\nabla(\varphi+i\psi)\cdot\nabla r + \tau(\Delta(\varphi + i\psi)) r - \frac{\kappa^2}{\mu_0} r =\tau^{-2}P a_2, \ \text{and} \\
-\Delta s + 2\tau\nabla(\varphi+i\psi)\cdot\nabla s + \tau(\Delta(\varphi + i\psi)) s  
+ \tau\frac{\lambda_0+\mu_0}{\lambda_0+2\mu_0}{\mu_0}^{\frac{1}{2}}\left[\nabla(\varphi+i\psi)\cdot r
- \nabla\cdot r\right]- \frac{\kappa^2}{\lambda_0+2\mu_0} s= \tau^{-2}P b_2.
\end{split}
\end{equation*}
Taking the $\nabla$ on the both sides of the above equations and applying the interior estimate on
\[
\begin{split}
-\Delta (\nabla r) = &- 2\tau\nabla(\nabla(\varphi+i\psi)\cdot\nabla r) - \tau\nabla((\Delta(\varphi + i\psi)) r) + \frac{\kappa^2}{\mu_0}\nabla r + \tau^{-2}\nabla(P a_2),  \\
\text{and}\hspace{1cm} -\Delta (\nabla s) 
 =& - 2\tau\nabla(\nabla(\varphi+i\psi)\cdot\nabla s) - \tau\nabla((\Delta(\varphi + i\psi)) s) \\
& - \tau\frac{\lambda_0+\mu_0}{\lambda_0+2\mu_0}{\mu_0}^{\frac{1}{2}}\left[\nabla(\nabla(\varphi+i\psi)\cdot r)
- \nabla(\nabla\cdot r)\right] + \frac{\kappa^2}{\lambda_0+2\mu_0}\nabla s + \tau^{-2}\nabla(P b_2),
\end{split}
\]
we obtain the estimate for the remainder term
$
 \Vert\nabla R\Vert_{H^2(\Omega_0)} \leq C,
$
where $C>0$ is a constant.
\begin{Case}{Linear Case.}\end{Case}
In the linear case, we choose $\varphi := -x\cdot\rho$ and $\psi := -x\cdot\rho^{\perp},$ where 
$\rho, \rho^{\perp}\in\mathbb{S}^2$ with $\rho\cdot\rho^{\perp} = 0.$
Then \eqref{conju_op} reduces to
\[
\begin{split}
 e^{-\tau(x\cdot\rho)}\tau^{-2}Pe^{\tau x\cdot(\rho+i\rho^{\bot})}(L_0 + \tau^{-1}L_1 +\tau^{-2}L_2) 
 = e^{i\tau(x\cdot\rho^{\bot})}
& [ \tau^{-1} T_{\rho}L_0 
 + \tau^{-2}\{T_{\rho}L_1 + PL_0\}   \\
& + \tau^{-3} \{T_{\rho}L_2 + PL_1\}  
 + \tau^{-4}(PL_2)] 
\end{split}
\]
where $T_{\rho} := -i(\rho + i\rho^{\bot})\cdot D + A_1(-i\rho+\rho^{\bot}).$ 
We choose $L_0, L_1$ and $L_2$ such that 
\begin{eqnarray}\label{C_R1}
  T_{\rho} L_0 = 0, 
\\ \label{C_R2}
 T_{\rho} L_1 + PL_0 = 0, 
\\ \label{C_R3}
   T_{\rho} L_2 + PL_1 = 0.
\end{eqnarray}
Here the equations \eqref{C_R1}, \eqref{C_R2}, \eqref{C_R3} are the system of Cauchy-Riemann type. Introducing new variable $z = (z_1, z_2) = (\rho + i\rho^{\bot})\cdot x$, the equation \eqref{C_R1}
becomes
\begin{equation}\label{C_R11}
 -2\partial_{\bar{z}} L_0 + A_1(-i\rho+\rho^{\bot}) L_0 = 0.
\end{equation}
Now we denote $L_0 :=(a_0, ~b_0)^\top$,
where $a_0 := (a_0^{1}, ~a_0^{2}, ~a_0^{3})^\top$
  $\in {\mathbb{C}}^3 $ and $b_0 \in \mathbb{C}. $
Therefore the equation \eqref{C_R11} reduces to
\begin{equation}\label{C_R12}
 \begin{cases}
  \partial_{\bar{z}}a_0 = 0 \\
  -2\partial_{\bar{z}}b_0 + \frac{\lambda_0 + \mu_0}{\lambda_0 + 2\mu_0}\mu_0^{\frac{1}{2}}(\rho+i\rho^{\bot})\cdot a_0 = 0. 
 \end{cases}
\end{equation}
In particular, any analytic function satisfies $\partial_{\bar{z}}a_0 = 0$. Take any non-zero constant vector $\rho$, then $\rho$ satisfies $\partial_{\bar{z}}\rho = 0$. Again if we replace $a_0$ by $\rho$ in the second equation 
of \eqref{C_R12} then we get a smooth solution $b_0$ on $\overline{\Omega}$. \\
Next, we look for $L_1$. The transport equation \eqref{C_R2} becomes
\begin{equation}\label{C_R21}
 -2\partial_{\bar{z}} L_1 + A_1(-i\rho+\rho^{\bot}) L_1 = -PL_0,
\end{equation}
where we denote $L_1 := (a_1, ~b_1)^\top$,
$a_1 :=(a_1^{1}, ~a_1^{2}, ~a_1^{3})^\top$
$\in {\mathbb{C}}^3 $ and $b_1 \in \mathbb{C}. $
Therefore the equation \eqref{C_R21} reduces to
\begin{equation}\label{stupid}
 -2\partial_{\bar{z}}  \left
     (\begin{array}{c} 
               a_1 \\ 
               b_1
     \end{array}
      \right) 
+  \left
     (\begin{array}{c} 
               0 \\ 
              \frac{\lambda_0 + \mu_0}{\lambda_0 + 2\mu_0}\mu_0^{\frac{1}{2}}(\rho+i\rho^{\bot})\cdot a_1
     \end{array}
      \right) 
= -P \left
     (\begin{array}{c} 
               a_0 \\ 
               b_0
     \end{array}
      \right).
\end{equation}

The equation \eqref{stupid} is solvable and the solution is a smooth function on $\overline{\Omega}$ since the right hand side is smooth, see for instance \cite{A-F}. 
We look for $L_2$ in a similar way. The transport equation \eqref{C_R3} becomes
\begin{equation}\label{C_R31}
 -2\partial_{\bar{z}} L_2 + A_1(-i\rho+\rho^{\bot}) L_2 = -PL_1
\end{equation} 
and then we obtain $L_2$ as a smooth function on $\overline{\Omega}$, 
where we denote $L_2 :=(a_2, ~b_2)^\top$,
with 
$a_2 :=(a_2^{1}, ~a_2^{2}, ~a_2^{3})^\top$
$ \in {\mathbb{C}}^3 $ and $b_2 \in \mathbb{C}. $  
Finally, to get the remainder term of the complex geometric optics solution, we choose $R \in{\mathbb{C}}^4,$ by Proposition \ref{4.1.3p1},
such that 
\begin{equation}\label{es_remainder0}
 e^{-\tau(x\cdot\rho)}\tau^{-2}Pe^{\tau x\cdot(\rho+i\rho^{\bot})}R =-e^{i\tau(x\cdot\rho^{\bot})}\tau^{-4}(PL_2), 
\end{equation}
hence
$
 \Vert R\Vert_{H_{scl}^{2}(\Omega)}\leq C \tau\Vert e^{i\tau(x\cdot\rho^{\bot})}\tau^{-4}(PL_2)\Vert_{L^2(\Omega)} = O(\tau^{-3}).
$
We can write the estimates for the remainder term as
\begin{equation}\label{es_rem1}
 \begin{cases}
  \Vert R\Vert_{L^2(\Omega)} \leq C\tau^{-3} \\
  \Vert \nabla R\Vert_{L^2(\Omega)} \leq C\tau^{-2} \\
  \Vert \nabla\nabla R\Vert_{L^2(\Omega)} \leq C\tau^{-1}. 
 \end{cases}
\end{equation}
To derive the estimate in $H^3(\Omega_0),$ $\Omega_0\subset\subset\Omega$ we proceed as in Case 1 using interior regularity estimates.
\end{proof}
\end{subsection}
\end{section}
\begin{section}{\textbf{Reconstruction Scheme}}
 In this section, we show how one can reconstruct some features of the obstacle using only one part of the farfield pattern. 
These features are extracted from the behavior of the indicator functions defined in section \ref{indi_sub} for large $\tau.$
Precisely, we can reconstruct the convex hull of the obstacle if we use CGOs with linear phases and some parts of its non-convex part if we use CGOs with logarithmic phase. 
Let us introduce the following two functions:
\begin{eqnarray}\label{linz}
h_D(\rho) &:=& \sup_{x\in D} x\cdot\rho, \ (\rho\in \mathbb{S}^2), \ \text{and}\
 \\ \label{vienna}
 d_D(x_0) &:=& \inf_{x\in D}\log|x-x_0|,\  (x_0 \in \mathbb{R}^3\setminus\overline{ch(\Omega)}).
\end{eqnarray}
\textbf{I. Use of purely $p$ or $s$ type CGOs} \\
 Recall that the indicator function $I_{ss}$ represents the energy 
when we use $s$ incident field and the $s$-part of the farfield data. Similarly, the indicator function $I_{pp}$ represents the energy when we use the $s$ incident field and the $p$-part of the farfield data.
So for this case we choose the $p$-part and $s$-part of the CGO solution discussed in subsection \ref{cgo_PP}. We
have the following theorem.
\begin{theorem}\label{theorem1}
\textbf{(CGOs with linear phase.)}
Let $\rho \in \mathbb{S}^{2}$ and take $v$ to be the $p$-type CGO of linear phase introduced in section \ref{cgo_PP}. Let $I_{pp}(\tau,t)$ be the corresponding indicator functions defined in \eqref{PP}. 
 For both the penetrable and the impenetrable cases \footnote{For the penetrable case, we further assume that $k^2$ is not the Dirichlet eigenvalue for the Lam\'e operator in $D$.
This is needed in section 5.2 where we use the single layer potential to represent the solution of the scattered field. This condition can be avoided by using combined single and double
layer potentials, as it is done for the acoustic case \cite{K-S-1}.}, we have the following characterizations of $h_{D}(\rho)$.
\begin{eqnarray}\label{main_behavior_1_pp}
 \vert \tau^{-1}I_{pp}(\tau,t)\vert \leq Ce^{-c\tau},\; \tau >>1, \; c, C>0, \mbox{and in particular,}  \lim_{\tau \to \infty}\vert I_{pp}(\tau,t)\vert = 0 \quad (t>h_D(\rho)),
\\ \label{main_behavior_2_pp}
 \liminf_{\tau \to \infty}|\tau^{-1}I_{pp}(\tau, h_{D}(\rho))| >0 \;, \mbox{ and precisely, }
 c \leq \tau^{-1}\vert I_{pp}(\tau,h_D(\rho))\vert \leq C \tau^2, 
\; \tau >>1,\; c, C>0,
\\ \label{main_behavior_3_pp}
 \vert \tau^{-1}I_{pp}(\tau,t) \vert \geq Ce^{c\tau}, \; \tau >>1,\; c, C>0
 , \mbox{ and in particular, } \lim_{\tau \to \infty}|I_{pp}(\tau,t)| = \infty \quad
 (t<h_D(\rho)).
\end{eqnarray}
\textbf{(CGOs with logarithmic phase.)}
Let $x_0\in \mathbb{R}^3\setminus\overline{ch(\Omega)}$ and set $v$ to be the $p$-type CGO with logarithmic phase introduced in section \ref{cgo_SS}. Let $I_{pp}(\tau,t)$ the
corresponding indicator functions defined in \eqref{PP}.
 For both the penetrable and the impenetrable cases, we have the following characterizations of $d_D(x_0)$.
\begin{eqnarray}\label{main_behavior_1_ppi}
 \vert \tau^{-1}I_{pp}(\tau,t)\vert \leq Ce^{-c\tau},\; \tau >>1, \; c, C>0, \mbox{and in particular,}  \lim_{\tau \to \infty}\vert I_{pp}(\tau,t)\vert = 0 \quad (t<d_D(x_0)),
\\ \label{main_behavior_2_ppi}
 \liminf_{\tau \to \infty}|\tau^{-1}I_{pp}(\tau, d_D(x_0))| >0 \;,  \mbox{ and precisely, }
 c \leq \tau^{-1}\vert I_{pp}(\tau,d_D(x_0))\vert \leq C \tau^2, 
\; \tau >>1,\; c,C>0,
\\ \label{main_behavior_3_ppi}
 \vert \tau^{-1}I_{pp}(\tau,t) \vert \geq Ce^{c\tau}, \; \tau >>1,\; c, C>0
 , \mbox{ and in particular, } \lim_{\tau \to \infty}|I_{pp}(\tau,t)| = \infty \quad
 (t>d_D(x_0)).
\end{eqnarray}
The above estimates are also valid if we replace $I_{pp}$ by $I_{ss}$ and the $p$-type CGOs by the $s$-type CGOs introduced in section \ref{cgo_PP} and section \ref{cgo_SPPS}.
\end{theorem} \\
\textbf{II. Use of CGOs with both non-vanishing $p$ and $s$ parts} \\
Regarding the indicator function $I_{sp},$ which depends on $s$-incident wave and $p$-part of the farfield and similarly for $I_{ps},$ which
depends on $p$-incident wave and $s$-part of the farfield, we cannot use the expression of the CGO solution as in \eqref{pp} and \eqref{ss}, since the $s$-part of \eqref{pp} and the $p$-part of \eqref{ss} are zero.
Instead, we use the CGOs we discussed in subsection \ref{cgo_SPPS}. Using these CGOs, we have the following theorem. 
\begin{theorem}\label{theorem2}
\textbf{(CGOs with linear phases.)}
Let $\rho \in \mathbb{S}^{2}$.
 For both the penetrable and the impenetrable cases \footnote{Same comments as for Theorem \ref{theorem1}.}, we have the following characterizations of $h_{D}(\rho)$.
\begin{eqnarray}\label{main_behavior_1_sp}
 \vert \tau^{-3}I_{sp}(\tau,t)\vert \leq Ce^{-c\tau},\; \tau >>1, \; c, C>0, \mbox{ and in particular, }  \lim_{\tau \to \infty}\vert I_{sp}(\tau,t)\vert = 0 \quad (t>h_D(\rho)),
\\ \label{main_behavior_2_sp}
 \liminf_{\tau \to \infty}|\tau^{-3}I_{sp}(\tau, h_{D}(\rho))| >0 \;,  \mbox{ and precisely, }
 c \leq \tau^{-1}\vert I_{sp}(\tau,h_D(\rho))\vert \leq C \tau^4, 
\; \tau >>1,\; c, C>0,
\\ \label{main_behavior_3_sp}
 \vert \tau^{-3}I_{sp}(\tau,t) \vert \geq Ce^{c\tau}, \; \tau >>1,\; c, C>0
 , \mbox{ and in particular, } \lim_{\tau \to \infty}|I_{sp}(\tau,t)| = \infty \quad
 (t<h_D(\rho)).
\end{eqnarray}
\textbf{(CGOs with logarithmic phases.)}
Let $x_0\in \mathbb{R}^3\setminus\overline{ch(\Omega)}$.
 For both the penetrable and the impenetrable cases, we have the following characterizations of $d_D(x_0)$.
\begin{eqnarray}\label{main_behavior_1_spi}
 \vert \tau^{-3}I_{sp}(\tau,t)\vert \leq Ce^{-c\tau}, \tau >>1, c, C>0, \mbox{ and in particular, }  \lim_{\tau \to \infty}\vert I_{sp}(\tau,t)\vert = 0 \quad (t<d_D(x_0)),
\\ \label{main_behavior_2_spi}
 \liminf_{\tau \to \infty}|\tau^{-3}I_{sp}(\tau, d_D(x_0))| >0 \;,  \mbox{ and precisely, }
 c \leq \tau^{-1}\vert I_{sp}(\tau,d_D(x_0))\vert \leq C \tau^4, 
\; \tau >>1,\; c, C>0,
\\ \label{main_behavior_3_spi}
 \vert \tau^{-3}I_{sp}(\tau,t) \vert \geq Ce^{c\tau}, \; \tau >>1,\; c, C>0
 , \mbox{ and in particular, } \lim_{\tau \to \infty}|I_{sp}(\tau,t)| = \infty \quad
 (t>d_D(x_0)).
\end{eqnarray}
The above estimates are valid if we replace $I_{sp}$ by $I_{ps}$.
\end{theorem}\\
From the above two theorems, we see that, in case of linear phase for a fixed direction $\rho$ (accordingly, in case of logarithmic phase for a fixed 
direction $x_0$), the behavior of the indicator function $I_{ij}$, where $ij = pp, ss, sp$ or $ps$, 
changes drastically
in terms of $\tau$: exponentially decaying if $t>h_D(\rho)$ (accordingly $t>d_D(x_0)$ for the logarithmic phase),
 polynomially behaving if $t=h_D(\rho)$ (accordingly $t=d_D(x_0)$ for logarithmic phase)
and exponentially growing if $t<h_D(\rho)$ (accordingly $t<d_D(x_0)$ for the logarithmic phase).
Using this property of the indicator functions we can reconstruct the support function $h_D(\rho), \rho\in\mathbb{S}^2$
(accordingly the distance $d_D(x_0), x_0\in\mathbb{R}^3\setminus\overline{ch(\Omega)}$, for the logarithmic phase)
from the farfield measurement. Finally, from this support function for linear phase, we can reconstruct the convex hull of D
and from the distance function for the logarithmic phase, we can, in addition to the convex hull, reconstruct parts of the non-convex part of the obstacle D.
Let us finish this section by rephrasing the formulas in Theorem \ref{theorem1} and Theorem \ref{theorem2} as follows:
$h_{D}(\rho)-t= \lim_{\tau \to \infty}
\frac{\log|I_{ij}(\tau,t)|}{2\tau},$ for $ij=pp,ss,sp,ps,$ when we use CGOs with linear phase and 
$t-d_D(x_0)= \lim_{\tau \to \infty}
\frac{\log|I_{ij}(\tau,t)|}{2\tau},$ for $ij=pp,ss,sp,ps,$ when we use CGOs with logarithmic phase. The formulas are easily deduced from \eqref{main_behavior_2_pp},\eqref{main_behavior_2_ppi},\eqref{main_behavior_2_sp} and \eqref{main_behavior_2_spi}
and the following identities
\begin{eqnarray}\label{lond}
 I_{ij}(\tau,t) = e^{2\tau(h_{D}(\rho)-t)} I_{ij}(\tau,h_{D}(\rho)),
\\ \label{newyo}
I_{ij}(\tau,t) = e^{2\tau(t-d_D(x_0))} I_{ij}(\tau,d_D(x_0)).
\end{eqnarray}
\end{section}
\section{Justification of the reconstruction schemes.}
In this section, we prove the above two theorems using all the CGOs for both the penetrable and impenetrable obstacle cases. 
For that we only focus the four points
\eqref{main_behavior_2_pp},\eqref{main_behavior_2_ppi},\eqref{main_behavior_2_sp} and \eqref{main_behavior_2_spi} since we have \eqref{lond} and \eqref{newyo}. In addition, the lower estimates in 
\eqref{main_behavior_2_pp},\eqref{main_behavior_2_ppi},\eqref{main_behavior_2_sp} and \eqref{main_behavior_2_spi} are the most difficult part since the upper bounds are easily
obtain using the Cauchy-Schwartz inequality, the well posedness of the forward problems and the upper estimate of the $H^1$-norms of the CGOs given in section 6.
So, we mainly focus in our proofs on the lower estimates.
\subsection{The penetrable obstacle case}
We consider $w$ as an incident field and $u^s(w)$ the scattered field, therefore the total field $\tilde w = w + u^s(w)$ satisfies the following problem
\begin{equation}  \label{Lame_Penetrable}
\begin{cases}
\nabla\cdot(\sigma(\tilde w)) + \kappa^2 \tilde w = 0, \; \mbox{ in } \mathbb{R}^3  \\
u^s(w) \ \ \text{satisfies the Kupradze radiation condition},
\end{cases}
\end{equation}
recalling that
$
 \sigma(\tilde w) = \lambda(\nabla\cdot\tilde w)I_3 + \mu(\nabla\tilde w + (\nabla\tilde w)^T).
$
The incident CGO field satisfies 
\begin{equation}\label{cgo_p}
 \nabla\cdot(\sigma_0(w)) + \kappa^2 w = 0 \ \ \text{in} \ \ \Omega,
\end{equation}
where 
$
 \sigma_0(v) = \lambda_0(\nabla\cdot v)I_3 + 2\mu_0\epsilon(v)
$
for any displacement field $v$. Accordingly, we will use $\sigma_D(v)$ to denote $\sigma(v)- \sigma_0(v),$ i.e.
$
 \sigma_D(v) = \lambda_D(\nabla\cdot v)I_3 + 2\mu_D\epsilon(v).
$
Note that for a matrix $A=(a_{ij})$, we use $\vert A\vert$ to denote $(\sum_{i,j}|a_{ij}|^{2})^{\frac{1}{2}}$. 
For any matrices $A=(a_{ij})$ and $B=(b_{ij})$ we define the product as follows 
\begin{equation}\label{trace}
 A\cdot B^T := tr(AB) 
\end{equation}
where $tr(A)$ is the trace of the matrix $A$.
Also frequently we will use the following basic identity:
\begin{equation}\label{trace_identity}
 \sigma(u)\cdot (\nabla v)^T = \sigma(v)\cdot (\nabla u)^T
\end{equation}
and Betti's identity
\begin{equation}\label{betti_identity}
 \int_{\Omega}(\nabla\cdot\sigma(u))\cdot v dx = -\int_{\Omega} \sigma(u)\cdot (\nabla v)^T dx + \int_{\partial \Omega}(\sigma(u)\cdot\nu)\cdot v ds(x).
\end{equation}
\begin{lemma}\label{lemma 4.1}
 We have the following identity
\[
 tr(\sigma(u)\nabla u) = \frac{3\lambda + 2\mu}{3}\vert\nabla\cdot u\vert^2 + 2\mu\vert\epsilon(u)-\frac{\nabla\cdot u}{3}I_3\vert^2.
\]
\end{lemma}
\begin{proof}
 Define $Sym \nabla u := \frac{1}{2}(\nabla u + (\nabla u)^T) = \epsilon(u).$ For any $\alpha, \beta$ and a matrix $A$, we have the following
identity, see \cite{Ikeha},
\begin{equation}\label{0p1}
 \alpha\vert tr(A)\vert^2 + 2\beta\vert Sym A\vert^2 = \frac{3\alpha + 2\beta}{3}\vert tr(A)\vert^2 + 2\beta\vert Sym A - \frac{tr(A)}{3}I_3\vert^2.
\end{equation}
 Substituting $A = \nabla u, \alpha = \lambda, \beta = \mu$ in \eqref{0p1}, we obtain
\[
 \begin{split}
  tr(\sigma(u)\nabla u)
& = \lambda \vert\nabla\cdot u\vert^2 + 2\mu\vert\epsilon(u)\vert^2 \\
& = \frac{3\lambda + 2\mu}{3}\vert\nabla\cdot u\vert^2 + 2\mu\vert\epsilon(u)-\frac{\nabla\cdot u}{3}I_3\vert^2.
 \end{split}
\]
\end{proof}\\
Let $v$ and $w$ be two incident waves. We set $I(v,w):=\int_D[u^s(v)\cdot(\overline{\sigma(w)}\cdot\nu)-\overline{v}\cdot(\sigma(u^s(w))\cdot\nu)]ds(x).$
Hence from \eqref{log1}, \eqref{log2}, \eqref{she1} and \eqref{she2}, we have 
\begin{equation}
I_{ss}(\tau,t)=I(v_s,v_s), \ I_{pp}(\tau,t)=I(v_p,v_p), \ I_{ps}(\tau,t)=I(v_p,v_s) \ \text{and}\ \ I_{sp}(\tau,t)=I(v_s,v_p).
\end{equation}
\subsubsection{Some key inequalities}
\ ~ \
\vspace{1mm}
\begin{lemma}\label{ss_1} Let $v$ and $w$ be two incident waves.
 We have the following estimates for the indicator function
\[
 I(v,w) = - \int_{\Omega}\sigma_D(w)\cdot (\nabla \bar v)^T dx - \int_{\Omega}\sigma_D(u^s(w))\cdot (\nabla \bar v)^T dx
\]
and
\begin{equation}\label{inss}
 \begin{split}
 -I(v,v) \geq
& \int_D \frac{4\mu_0\mu_D}{3\mu} \vert \epsilon(v)\vert^2 dx - \int_D \frac{4\mu_0\mu_D}{9\mu} \vert (\nabla\cdot v)I_3\vert^2 dx - \kappa^2\int_{\Omega} \vert u^s(v)\vert^2dx \\
& - \int_{\partial\Omega} (\sigma(u^s(v))\cdot \nu)\cdot\overline{u^s(v)}ds(x). \\
\end{split}
\end{equation}
\end{lemma}
\begin{proof}
 We have $\tilde w= w + u^s(w)$ and $\tilde v= v + u^s(v)$. So, from Betti's identity we obtain
\begin{equation}\label{p1}
\begin{split}
  \int_{\partial \Omega}(\sigma(\tilde w)\cdot\nu)\cdot \overline{v} ds(x) 
& = \int_{\Omega}(\nabla\cdot\sigma(\tilde w))\cdot \overline{v} dx + \int_{\Omega} \sigma(\tilde w)\cdot (\nabla\overline{v})^T dx \\
 & = -\kappa^2\int_{\Omega} \tilde w\cdot\overline{v} + \int_{\Omega} \sigma(\tilde w)\cdot (\nabla \overline{v})^T dx.
\end{split}
\end{equation}
On the other hand, again applying Betti's identity we have
\begin{equation}\label{p2}
 \begin{split}
\int_{\partial\Omega}(\overline{\sigma(v)}\cdot\nu)\cdot\tilde w ds(x)
& = \int_{\partial\Omega} (\overline{\sigma_0(v)}\cdot\nu)\cdot\tilde w ds(x)
 = \int_{\Omega}(\nabla\cdot\overline{\sigma_0(v)})\cdot\tilde w dx + \int_{\Omega}\overline{\sigma_0(v)}\cdot(\nabla\tilde w)^T dx \\
& = -\kappa^2 \int_{\Omega} \overline{v}\cdot\tilde w dx + \int_{\Omega}\overline{\sigma_0(v)}\cdot (\nabla \tilde w)^T dx 
\end{split}
\end{equation}
Subtracting \eqref{p1} and \eqref{p2} gives
\begin{equation}\label{p3}
 \begin{split}
  \int_{\partial \Omega}(\sigma(\tilde w)\cdot\nu)\cdot \overline{v} ds(x) - \int_{\partial\Omega}(\overline{\sigma(v)}\cdot\nu)\cdot\tilde w ds(x)
& = \int_{\Omega}(\sigma(\tilde w)- \sigma_0(\tilde w)) \cdot (\nabla \overline{v})^T dx\\
& = \int_{\Omega} \sigma_D(\tilde w)\cdot(\nabla \overline{v})^T dx
 = \int_{\Omega} \overline{\sigma_D(v)}\cdot(\nabla \tilde w)^T dx.
 \end{split}
\end{equation}
Therefore replacing $\tilde w$ by $w + u^s(w)$ in \eqref{p3}, we obtain
\begin{equation}\label{p4}
\begin{split}
& I(v,w) \\
& := \int_{\partial\Omega}u^s(w)\cdot(\overline{\sigma(v)}\cdot\nu)ds(x) - \int_{\partial\Omega} \overline{v}\cdot(\sigma(u^s(w))\cdot\nu)ds(x) \\
& = \int_{\partial\Omega}(\sigma(w)\cdot\nu)\cdot \overline{v}ds(x)-\int_{\partial\Omega}(\overline{\sigma(v)}\cdot\nu)\cdot w ds(x)-\int_{\Omega}\sigma_D(w)\cdot (\nabla \overline{v})^Tdx-\int_{\Omega}\sigma_D(u^s(w))\cdot(\nabla \overline{v})^Tdx.
\end{split}
 \end{equation}
Also $v$ and $w$ satisfies the following equations
\begin{eqnarray}\label{1L4.2}
 \nabla\cdot\sigma_0(v) + \kappa^2 v = 0,
\\ \label{2L4.2}
 \nabla\cdot\sigma_0(w) + \kappa^2 w = 0.
\end{eqnarray}
Now, multiplying \eqref{1L4.2} by $w$, \eqref{2L4.2} by $\bar v$ and doing integration by parts we end up with
\begin{eqnarray}\label{3L4.2}
 -\int_{\Omega}\overline{\sigma_0(v)}\cdot(\nabla w)^T dx + \int_{\partial\Omega}(\overline{\sigma_0(v)}\cdot\nu)\cdot w ds(x) + \kappa^2\int_{\Omega}\overline{v}\cdot w dx &=& 0, \ \text{and}\
\\ \label{4L4.2}
 -\int_{\Omega}\overline{\sigma_0(v)}\cdot(\nabla w)^T dx + \int_{\partial\Omega}(\sigma_0(w)\cdot\nu)\cdot \overline{v} ds(x) + \kappa^2\int_{\Omega}\overline{v}\cdot w dx &=& 0.
\end{eqnarray} 
Subtracting \eqref{3L4.2} and \eqref{4L4.2} we obtain
\begin{equation}\label{5L4.2}
\int_{\partial\Omega}(\sigma_0(w)\cdot\nu)\cdot \overline{v} ds(x)- \int_{\partial\Omega}(\overline{\sigma_0(v)}\cdot\nu)\cdot w ds(x)= 0.
\end{equation}
Therefore substituting \eqref{5L4.2} in \eqref{p4} we have the formula for $I(v,w)$
\begin{equation}\label{6L4.2}
 I(v,w) = -\int_{\Omega}\sigma_D(w)\cdot (\nabla \overline{v})^Tdx-\int_{\Omega}\sigma_D(u^s(w))\cdot(\nabla \overline{v})^Tdx.
\end{equation}
 Now we look for the estimate of $I(v,v).$ Replacing $w$ by $v$, we have
\begin{equation}\label{p5}
 \begin{split}
I(v,v) 
& = -\int_{\Omega}\sigma_D(v)\cdot (\nabla \overline{v})^Tdx-\int_{\Omega}\sigma_D(u^s(v))\cdot(\nabla \overline{v})^Tdx 
 = -\int_{\Omega}\sigma_D(\tilde v)\cdot (\nabla \overline{v})^Tdx \\
 &= -\int_{\Omega}\overline{\sigma_D(v)}\cdot (\nabla \tilde v)^Tdx 
 = -\int_{\Omega}\overline{\sigma_D(\tilde v)}\cdot (\nabla \tilde v)^Tdx+\int_{\Omega}\overline{\sigma_D(u^s(v))}\cdot(\nabla \tilde v)^Tdx  \\
& = -\int_{\Omega}\overline{\sigma_D(\tilde v)}\cdot (\nabla\tilde v)^Tdx+\int_{\Omega}\sigma_D(\tilde v)\cdot(\nabla \overline{u^s(v)})^Tdx.
\end{split}
\end{equation}
On the other hand we have
\[
  \kappa^2\int_{\Omega}\vert u^s(v)\vert^2 dx 
 = \int_{\Omega}(\sigma(\tilde v)-\sigma_0(v))\cdot(\nabla \overline{u^s(v)})^Tdx - \int_{\partial\Omega}((\sigma(\tilde v)-\sigma_0(v))\cdot\nu)\cdot \overline{u^s(v)}ds(x).  
\]
Substituting $v=\tilde v- u^s(v)$ in the 1st term and $\tilde v= v+ u^s(v)$ in the 2nd term in right side of the above identity, we obtain
\begin{equation}\label{p6}
\begin{split}
  \kappa^2\int_{\Omega}\vert u^s(v)\vert^2 dx 
& = \int_{\Omega}(\sigma(\tilde v)-\sigma_0(\tilde v)+ \sigma_0(u^s(v)))\cdot(\nabla \overline{u^s(v)})^Tdx - \int_{\partial\Omega}(\sigma(u^s(v))\cdot\nu)\cdot \overline{u^s(v)}ds(x) \\
& = \int_{\Omega}\sigma_D(\tilde v)\cdot(\nabla \overline{u^s(v)})^Tdx + \int_{\Omega}\sigma_0(u^s(v))\cdot(\nabla \overline{u^s(v)})^Tdx - \int_{\partial\Omega}(\sigma(u^s(v)\cdot\nu))\cdot \overline{u^s(v)}ds(x).
\end{split}
\end{equation}
Combining \eqref{p5} and \eqref{p6}, we obtain
\begin{equation}\label{p7}
\begin{split}
 -I(v,v) 
& = \int_{\Omega}\overline{\sigma_D(\tilde v)}\cdot(\nabla \tilde v)^T dx + \int_{\Omega}\sigma_0(u^s(v))\cdot(\nabla \overline{u^s(v)})^Tdx - \kappa^2\int_{\Omega}\vert u^s(v)\vert^2dx \\
&- \int_{\partial\Omega}(\sigma(u^s(v))\cdot\nu)\cdot \overline{u^s(v)}ds(x). 
\end{split}
\end{equation}
Therefore from Lemma \ref{lemma 4.1}, we have
\begin{equation}\label{p8}
\begin{split}
 -I(v,v) =
 &  \int_{\Omega}\frac{3\lambda_D + 2\mu_D}{3}\vert\nabla\cdot \tilde v\vert^2dx + \int_{\Omega}2\mu_D\vert\epsilon(\tilde v)-\frac{\nabla\cdot \tilde v}{3}I_3\vert^2dx 
  + \int_{\Omega}\frac{3\lambda_0 + 2\mu_0}{3}\vert\nabla\cdot u^s(v)\vert^2dx \\
&+ \int_{\Omega}2\mu_0\vert\epsilon(u^s(v))-\frac{\nabla\cdot u^s(v)}{3}I_3\vert^2dx 
  - \kappa^2\int_{\Omega}\vert u^s(v)\vert^2dx - \int_{\partial\Omega}(\sigma(u^s(v))\cdot\nu)\cdot \overline{u^s(v)}ds(x).
\end{split}
\end{equation}
To estimate the first four integrals in the right hand side of \eqref{p8}, we follow few steps from [\cite{Ikeh}, Proposition 5.1].
Set $B_1 = \epsilon(v) - \frac{\nabla\cdot v}{3}I_3$ and $B_2 = \epsilon(\tilde v) - \frac{\nabla\cdot\tilde v}{3}I_3$.
So,
\[
 \begin{split}
& \frac{3\lambda_D + 2\mu_D}{3}\vert\nabla\cdot \tilde v\vert^2 + 2\mu_D\vert\epsilon(\tilde v)-\frac{\nabla\cdot \tilde v}{3}I_3\vert^2 
 + \frac{3\lambda_0 + 2\mu_0}{3}\vert\nabla\cdot u^s(v)\vert^2 + 2\mu_0\vert\epsilon(u^s(v))-\frac{\nabla\cdot u^s(v)}{3}I_3\vert^2 \\
& = \frac{3\lambda_D + 2\mu_D}{3}\vert\nabla\cdot\tilde v\vert^2 + 2\mu_D\vert B_2\vert^2 + \frac{3\lambda_0 + 2\mu_0}{3}\vert\nabla\cdot u^s(v)\vert^2 + 2\mu_0\vert B_1-B_2\vert^2 \\
& = \frac{3\lambda + 2\mu}{3}\vert\nabla\cdot\tilde v\vert^2-\frac{2(3\lambda_0 + 2\mu_0)}{3}(\nabla\cdot v)(\nabla\cdot\tilde v)
+ 2\mu\vert B_2\vert^2 - 2\mu_0 B_1\cdot B_2 + \frac{3\lambda_0 + 2\mu_0}{3}\vert\nabla\cdot v\vert^2 + 2\mu_0\vert B_1\vert^2 \\
& = \vert\sqrt{\frac{3\lambda + 2\mu}{3}}\nabla\cdot\tilde v - \sqrt{\frac{3}{3\lambda + 2\mu}}\frac{3\lambda_0 + 2\mu_0}{3}\nabla\cdot v\vert^2  + \vert\sqrt{2\mu}B_2 - \frac{2\mu_0}{\sqrt{2\mu}}B_1\vert^2 \\
& + \left\{ \frac{3\lambda_0 + 2\mu_0}{3} - \frac{3}{3\lambda + 2\mu}\left(\frac{3\lambda_0 + 2\mu_0}{3}\right)^2\right\}\vert \nabla\cdot v\vert^2 + \left\{2\mu_0 - \frac{(2\mu_0)2}{2\mu}\right\}\vert B_1\vert^2 \\
& \geq \frac{3\lambda_0 + 2\mu_0}{3(3\lambda + 2\mu)}\left\{3\lambda_D + 2\mu_D\right\} \vert \nabla\cdot v\vert^2 + \frac{2\mu_0\mu_D}{\mu}\vert \epsilon(v) - \frac{\nabla\cdot v}{3}I_3\vert^2.
\end{split}
\]
%
%
%
Hence
\begin{equation} \label{mou}
\begin{split}
 -I(v,v) 
& \geq \int_D \left\{\frac{3\lambda_0 + 2\mu_0}{3(3\lambda + 2\mu)}(3\lambda_D + 2\mu_D) \vert \nabla\cdot v\vert^2 + \frac{2\mu_0\mu_D}{\mu}\vert \epsilon(v) - \frac{\nabla\cdot v}{3}I_3\vert^2
\right\}dx \\
& -\kappa^2\int_{\Omega}\vert u^s(v)\vert^2dx - \int_{\partial\Omega}(\sigma(u^s(v))\cdot \nu)\cdot \overline{u^s(v)}ds(x).
\end{split}
\end{equation}
Therefore,
\begin{equation*}
 \begin{split}
&-I(v,v) \\
& \geq \int_D\frac{2\mu_0\mu_D}{\mu}\vert \epsilon(v) - \frac{\nabla\cdot v}{3}I_3\vert^2dx - \kappa^2\int_{\Omega}\vert u^s(v)\vert^2dx - \int_{\partial\Omega}(\sigma(u^s(v))\cdot \nu)\cdot \overline{u^s(v)}ds(x) \\
& \geq \int_D \frac{4\mu_0\mu_D}{3\mu} \vert \epsilon(v)\vert^2 dx - \int_D \frac{4\mu_0\mu_D}{9\mu} \vert (\nabla\cdot v)I_3\vert^2 dx - \kappa^2\int_{\Omega} \vert u^s(v)\vert^2dx - \int_{\partial\Omega} (\sigma(u^s(v))\cdot \nu)\cdot \overline{u^s(v)}ds(x).
\end{split}
\end{equation*}
\end{proof}
\begin{lemma}\label{4.3lem}
 We have 
\begin{equation}\label{4.3lem0}
 |\int_{\partial\Omega}(\sigma(u^s(v))\cdot\nu)\cdot \overline{u^s(v)}ds(x)| \leq C\mathcal{F}\Vert\nabla v\Vert_{L^2(D)}^{2}
\end{equation}
where $v$ is the incident field, $u^s(v)$ is the scattered field and $\mathcal{F}$ is defined by 
\[
 \mathcal{F} := \int_{B\setminus\overline{\Omega}}\Vert(\nabla\Phi(x,\cdot))^T\Vert_{L^2(D)}^{2}dx + \int_{B\setminus\overline{\Omega}}\Vert\nabla(\nabla\Phi(x,\cdot))^T\Vert_{L^2(D)}^{2}dx,
\]
with $B$ as any smooth domain containing $\overline{\Omega}$.
\end{lemma}
\begin{proof}
 Let us take a ball $\tilde B$ such that $D\subset\subset\tilde B \subset\subset\Omega$. The scattering field $u^s(v)$ satisfies
\begin{equation}\label{4.3lem1}
 \begin{cases}
  \nabla\cdot\sigma(u^s(v)) + \kappa^2 u^s(v) = -\nabla\cdot\sigma_D(v), \; \mbox{ in } \mathbb{R}^3 \\
   u^s(v) \ \text{satisfies radiation condition} 
 \end{cases}
\end{equation}
and $\Phi(\cdot,\cdot)$ be the fundamental tensor for elasticity satisfies
\begin{equation}\label{4.3lem2}
 \nabla\cdot\sigma(\Phi(x,y)) + \kappa^2 \Phi(x,y) = \delta_x(y), \; \mbox{ in } \mathbb{R}^3
\end{equation}
where $\delta_x$ is the Dirac measure at $x$.\\
  \textbf{Step 1}\\
  First we show that for $x\in\mathbb{R}^3\setminus \overline{\tilde B}$, we have
\begin{equation}\label{4.3lem6}
 u^s(v)(x) = \int_D\sigma_D(v(y))\cdot(\nabla\Phi(x,y))^Tdy.
\end{equation}
 Indeed, take $x\in \mathbb{R}^3\setminus \overline{\tilde B}$. Then multiplying by $\Phi(x,y)$ in \eqref{4.3lem1} and doing integration by parts we obtain,
\begin{equation}
\begin{split}
& \int_{\tilde B} \sigma_D(v)\cdot (\nabla\Phi(x,y))^T dy  \\ 
& = - \int_{\tilde B}\sigma(u^s(v))\cdot(\nabla\Phi(x,y))^Tdy + \int_{\partial {\tilde B}} (\sigma(u^s(v))\cdot\nu)\cdot \Phi(x,y) ds(y) + \kappa^2\int_{\tilde B} u^s(v)\cdot\Phi(x,y)dy \\
& = - \int_{\tilde B}\sigma(\nabla\Phi(x,y))\cdot(u^s(v))^Tdy + \int_{\partial {\tilde B}} (\sigma(u^s(v))\cdot\nu)\cdot \Phi(x,y) ds(y) + \kappa^2\int_{\tilde B} u^s(v)\cdot\Phi(x,y)dy.
\end{split}
\end{equation}
Applying integration by parts and from \eqref{4.3lem2}, we have
\begin{equation}\label{4.3lem3}
 \int_{\tilde B} \sigma_D(v)\cdot (\nabla\Phi(x,y))^T dy = \int_{\partial {\tilde B}}(\sigma(u^s(v))\cdot\nu)\cdot\Phi(x,y)ds(y) - \int_{\partial {\tilde B}}(\sigma(\Phi(x,y))\cdot\nu)\cdot u^s(v)ds(y).
\end{equation}
On the other hand, consider a ball $B_R$ such that ${\tilde B} \subset B_R$ and take $x\in B_R\setminus\overline{\tilde B}.$ Now multiplying both sides in \eqref{4.3lem1} by $\Phi(x,y)$ and applying integration by parts, we obtain
\begin{equation}\label{4.3lem4}
\begin{split}
& -\int_{B_R\setminus(\overline{\tilde B}\cup\overline{B_{\epsilon}})} \sigma(\Phi(x,y))\cdot(\nabla u^s(v))^Tdy \\
 & + \int_{\partial {\tilde B} \cup \partial B_{\epsilon}} (\sigma(u^s(v))\cdot\nu)\cdot\Phi(x,y)ds(y)
+ \kappa^2 \int_{B_R\setminus(\overline{\tilde B}\cup\overline{B_{\epsilon}})} u^s(v)\cdot\Phi(x,y)dy = 0.
\end{split}
\end{equation}
Again doing integration by parts on the 1st term of \eqref{4.3lem4} and from \eqref{4.3lem2}, we have
\begin{equation}\label{4.3lem5}
 \int_{\partial {\tilde B}\cup\partial B_{\epsilon}}(\sigma(u^s(v))\cdot\nu)\cdot\Phi(x,y)ds(y) - \int_{\partial{\tilde B}\cup\partial B_{\epsilon}} (\sigma(\Phi(x,y))\cdot\nu)\cdot u^s(v)ds(y) = 0.
\end{equation}
Combining \eqref{4.3lem3} and \eqref{4.3lem5} together with the following relation, see for instance \cite{CK}
\[
 \lim_{\epsilon\rightarrow 0}\left[\int_{\partial B_{\epsilon}} (\sigma(\Phi(x,y))\cdot\nu)\cdot u^s(v)ds(y) - \int_{\partial B_{\epsilon}}(\sigma(u^s(v))\cdot\nu)\cdot\Phi(x,y)ds(y)\right] = u^s(v)(x)
\]
Hence, we obtain \eqref{4.3lem6}.\\
 \textbf{Step 2}\\
    We justify the estimate \eqref{4.3lem0}. Consider a smooth domain $B$ such that $\Omega\subset\subset B.$ Using Cauchy-Schwartz inequality and trace theorem we have
\begin{eqnarray}\label{4.3lem7}
 \vert \int_{\partial\Omega}(\sigma(u^s(v))\cdot\nu)\cdot \overline{u^s(v)}ds(x)\vert  
 \leq \Vert u^s(v)\Vert_{H^{\frac{1}{2}}(\partial\Omega)} \Vert \sigma(u^s(v))\cdot\nu\Vert_{H^{-\frac{1}{2}}(\partial\Omega)} 
 \leq C \Vert u^s(v)\Vert_{H^1(B\setminus \overline\Omega)}^{2}.
\end{eqnarray}
On the other hand, from \eqref{4.3lem6} using Cauchy-Schwartz inequality, we obtain for $x\in B\setminus\overline{\Omega}$
\begin{eqnarray*}
 \vert u^s(v)(x)\vert &\leq& \Vert\sigma_D(v)\Vert_{L^2(D)}\Vert(\nabla\Phi(x,\cdot))^T\Vert_{L^2(D)} 
\\ \vert\nabla u^s(v)(x)\vert &\leq& \Vert\sigma_D(v)\Vert_{L^2(D)}\Vert(\nabla(\nabla\Phi(x,\cdot))^T)\Vert_{L^2(D)}.
\end{eqnarray*}
Therefore from \eqref{4.3lem7} together with the estimate
\[
 \Vert u^s(v)\Vert_{H^1(B\setminus\overline{\Omega})}^{2} \leq \Vert\sigma_D(v)\Vert_{L^2(D)}^{2} \left\{\int_{B\setminus\overline{\Omega}}\Vert(\nabla\Phi(x,\cdot))^T\Vert_{L^2(D)}^{2}dx + \int_{B\setminus\overline{\Omega}}\Vert\nabla(\nabla\Phi(x,\cdot))^T\Vert_{L^2(D)}^{2}dx\right\}
\]
 we obtain our required estimate \eqref{4.3lem0}.
\end{proof}
\begin{lemma}\label{Lpest}($L^2-L^q$-estimate). 
There exists $1\leq q_0 < 2$ such that for $q_0 < q \leq 2,$
\[
 \Vert u^s(v)\Vert_{L^2(\Omega)} \leq C \Vert \nabla v\Vert_{L^q(D)}
\]
with a positive constant $C$.
\end{lemma}
\begin{proof}
We recall that $u^s(v)$ satisfies
\begin{equation}  \label{Lame_imPenetrab}
\begin{cases}
\nabla\cdot\sigma(u^s(v)) + \kappa^2 u^s(v) = -\nabla\cdot\sigma_D(v), \; \mbox{ in } \Omega  \\
u^s(v) = u^s(v)|_{\partial\Omega}, \ \text{on} \ \partial\Omega.
\end{cases}
\end{equation}
We write it as $u^s(v)=u_{1}^{s}(v)+u_{2}^{s}(v)$ where $u_{1}^{s}(v)$ satisfies
\begin{equation}  \label{lili1}
\begin{cases}
\nabla\cdot\sigma(u_{1}^{s}(v)) + \kappa^2 u_{1}^{s}(v) = 0, \; \mbox{ in } \Omega  \\
u_{1}^{s}(v) = u^s(v)|_{\partial\Omega}, \ \text{on} \ \partial\Omega
\end{cases}
\end{equation}
and $u_{2}^{s}(v)$ satisfies
\begin{equation}  \label{lili2}
\begin{cases}
\nabla\cdot\sigma(u_{2}^{s}(v)) + \kappa^2 u_{2}^{s}(v) = -\nabla\cdot\sigma_D(v), \; \mbox{ in } \Omega  \\
u_{2}^{s}(v) = 0, \ \text{on} \ \partial\Omega.
\end{cases}
\end{equation}
\textbf{Step 1:}\\
There exists a positive constant $C$, such that for every $q>1,$ we have
\begin{equation}\label{lali}
\|u_{1}^{s}(v)\|_{L^2(\Omega)} \leq C \|\nabla v\|_{L^q(D)}.
\end{equation}
Indeed, from \eqref{4.3lem6}, we deduce that
$
 \|u^s(v)\|_{H^{\frac{1}{2}}(\partial\Omega)} \leq C \|\nabla v\|_{L^q(D)},
$
since $D \subset\subset\Omega$.
The estimate \eqref{lali} comes now from the well posedness of \eqref{lili1}.\\
\textbf{Step 2:}\\
There exists $1\leq q_0<2$ such that for $q_0<q\leq 2,$ we have 
\[
 \|u_{2}^{s}(v)\|_{L^2(\Omega)} \leq C \|\nabla v\|_{L^q(D)}.
\]
Indeed, let $\textbf{u} \in H_{0}^{1}(\Omega)$ be the distribution solution satisfying the elastic model
\[
 \nabla\cdot(\sigma(\textbf{u})) + \kappa^2\textbf{u} = \overline{u_{2}^{s}(v)} \ \ \text{in} \ \ \Omega.
\]
Multiplying both sides of the last equation by $u_{2}^{s}(v)$ and integrating by parts, we obtain
\begin{equation}\label{Lp_es1}
 \begin{split}
\Vert{u_{2}^{s}(v)}\Vert_{L^2(\Omega)}^{2} 
& = -\int_{\Omega} tr(\sigma(\textbf{u})\nabla u_{2}^{s}(v)) dx + \kappa^2\int_{\Omega} \textbf{u}\cdot u_{2}^{s}(v)dx \\   
& = -\int_{\Omega} tr(\sigma(u_{2}^{s}(v))\nabla \textbf{u})dx + \kappa^2 \int_{\Omega} \textbf{u}\cdot u_{2}^{s}(v)dx.
\end{split}
\end{equation}
Similarly, multiplying both sides of the equation
$
 \nabla\cdot\sigma(u_{2}^{s}(v)) + \kappa^2 u_{2}^{s}(v) = -\nabla\cdot\sigma_{D}(v)
$
by $\textbf{u}$ and integrating by parts, we have
\begin{equation}\label{Lp_es2}
 -\int_{\Omega} tr(\sigma(u_{2}^{s}(v))\nabla\textbf{u})dx + \kappa^2\int_{\Omega}u_{2}^{s}(v)\cdot\textbf{u}dx = \int_{\Omega} tr(\sigma_D(v)\nabla\textbf{u})dx.
\end{equation}
From \eqref{Lp_es1} and \eqref{Lp_es2} we get
\[
 \Vert{u_{2}^{s}(v)}\Vert_{L^2(\Omega)}^{2} = \int_{\Omega} tr(\sigma_{D}(v)\nabla\textbf{u})dx.
\]
Then applying H{\"o}lder's inequality, for any $1\leq q' <\infty$, we obtain
\begin{equation}\label{Lp_es3}
 \Vert {u_{2}^{s}(v)}\Vert_{L^2(\Omega)}^{2} \leq \Vert\sigma_D(v)\Vert_{L^q(D)}\Vert\nabla\textbf{u}\Vert_{L^{q'}(\Omega)},
\end{equation}
where $\frac{1}{q} + \frac{1}{q'} = 1.$
By definition of $\textbf{u}$ we have
\begin{equation}\label{Mann}
 \begin{cases}
 & \nabla\cdot(\sigma(\textbf{u})) = \overline{u^s(v)} - \kappa^2 \textbf{u} \ \ \text{in} \ \ \Omega, \\
 &  \textbf{u} = 0 \ \ \text{on} \ \ \partial\Omega. 
\end{cases}
\end{equation}
In \cite{R-C-G}, the problem \eqref{Mann} is studied in the $L^p(\Omega)$ spaces and it is proved that $\exists q'_{1}>2$ and $C>0$ such that
$\forall 2\leq q' < q'_{1},$ we have
\[
 \|\textbf{u}\|_{W^{1,q'}(\Omega)} \leq C \|u_{2}^{s}(v)-\kappa^2\textbf{u}\|_{W^{-1,q'}(\Omega)}.
\]
Since $L^2(\Omega) \hookrightarrow W^{-1,q'}(\Omega)$ for $2\leq q'\leq 6,$ 
then we deduce that
\begin{equation}\label{Lp_es4}
 \Vert\nabla\textbf{u}\Vert_{L^{q'}(\Omega)} \leq C \{\Vert\textbf{u}\Vert_{L^2(\Omega)} + \Vert u_{2}^{s}(v)\Vert_{L^2(\Omega)}\}
\end{equation}
for some $C > 0.$
From the well posedness in $H^1(\Omega),$ we have the $L^2-$estimate
\begin{equation}\label{Lp_es5}
 \Vert\textbf{u}\Vert_{L^2(\Omega)} \leq C \Vert u_{2}^{s}(v)\Vert_{L^2(\Omega)}.
\end{equation}
Therefore from \eqref{Lp_es4} and \eqref{Lp_es5} we have
\begin{equation}\label{Lp_es6}
 \Vert\nabla\textbf{u}\Vert_{L^{q'}(\Omega)} \leq C \Vert u_{2}^{s}(v)\Vert_{L^2(\Omega)}
\end{equation}
 for some $C > 0.$ Combining \eqref{Lp_es3} and \eqref{Lp_es6}, we obtain
$
 \Vert u_{2}^{s}(v)\Vert_{L^2(\Omega)} \leq C \Vert\nabla v\Vert_{L^q(D)}
$
for some $C = C(\kappa,\lambda,\mu)$, for $q_0<q\leq2$, where $\frac{1}{q_0}+\frac{1}{q'_0}=1$ with $q'_0=\min\{q'_1,6\}.$ 
\end{proof}\\
Applying the Korn's inequality, the 1st term of the right hand side of the inequality \eqref{inss} can be lower bounded by
the norms CGO solutions only. In $(p,p)$ case, the 2nd term behaves like the CGO solution 
of the Helmholtz equation and for $(s,s)$ case it will be zero. Therefore, using Lemma \ref{4.3lem} and Lemma \ref{Lpest}, then from the inequality \eqref{inss} 
we deduce that
\begin{equation}\label{gen_pp}
   -I(v,v) 
 \geq (C_1-\mathcal{F})\Vert \nabla v\Vert_{L^2(D)}^{2} + c_1\Vert v\Vert_{L^2(D)}^{2} - c_2\|V\|_{L^2(D)}^{2} - c_3\Vert\nabla v\Vert_{L^q(D)}^{2}
\end{equation}
where $q_0<q<2$ and $v$ is considered to be $u_p$, $v=u_p=\nabla V$ with $V$ as the CGOs satisfying $(\Delta + \kappa_{p}^{2})V =0$ and 
\begin{equation}\label{gen_ss}
   -I(v,v) 
 \geq (C_1-\mathcal{F})\Vert \nabla v\Vert_{L^2(D)}^{2} + c_1\Vert v\Vert_{L^2(D)}^{2} - c_3\Vert\nabla v\Vert_{L^q(D)}^{2}
\end{equation}
where $v$ is considered to be $u_s$, $v=u_s:=\curl W,$ $W$ as the CGOs satisfying $(\Delta+\kappa_{s}^{2})W=0$. \\
 For the mixed case the indicator function can be written of the form
\[
 \begin{split}
 & I(u_s,u_p) = - I(u_s,u_s) + I(u_s,u), \\ 
 & I(u_p,u_s) = - I(u_s,u_s) + I(u,u_s),
\end{split}
\]
where CGO $u = u_p+u_s.$ Using the Cauchy's $\epsilon$-inequality, we have
\[
 |I(u_s,u)| \leq C\left(\frac{1}{\epsilon}\|\nabla u\|_{L^2(D)}^{2} + \epsilon\|\nabla u_s\|_{L^2(D)}^{2} \right),
\]
then from Lemma \ref{ss_1} and Lemma \ref{4.3lem}, we obtain for $q_0<q<2,$
\begin{equation}\label{sp_linear}
 \begin{split}
&  I(u_s,u_p), I(u_p,u_s) \\
& \geq (C - \mathcal{F} - {\tilde C}{\epsilon})\|\nabla u_s\|_{L^2(D)}^{2} + (C - \mathcal{F} - c_2)\|u_s\|_{L^2(D)}^{2} - c_1\Vert\nabla u_s\Vert_{L^q(D)}^{2} - \frac{\tilde{C}}{\epsilon} \|\nabla u\|_{L^2(D)}^{2}.
 \end{split}
\end{equation}
\subsubsection{Proof of Theorem \ref{theorem1}}
Assume $v$ to be the $p$-part of the CGOs with linear and logarithmic phases.
  From \eqref{gen_pp} together with Lemma \ref{ne2}
 we obtain
\[
\begin{split}
   \frac{-\tau^{-1}I(v,v)}{\Vert v\Vert_{L^2(D)}^{2}} 
& \geq \tau^{-1}\left[(C_1-\mathcal{F})\frac{\|\nabla v\|_{L^2(D)}^{2}}{\Vert v\Vert_{L^2(D)}^{2}} + c_1 - c_2\frac{\|V\|_{L^2(D)}^{2}}{\Vert v\Vert_{L^2(D)}^{2}} - \kappa^2\frac{\Vert\nabla v\Vert_{L^q(\Omega)}^{2}}{\Vert v\Vert_{L^2(D)}^{2}}\right].  \\
& \geq \tau^{-1}\left[(C_1-\mathcal{F})c\tau^2 + c_1 -c_2\tau^{-2} - c_3\tau^{3-\frac{2}{q}}\right].\\
\end{split}
\]
Therefore
$
 -\tau^{-1}I(v,v) \geq (C_1-\mathcal{F})c, \ \ \tau\gg1, \ \ c>0.
$
Remark that $B$ is taken to be arbitrary (but containing $\Omega$). Hence,
we choose it such that $\partial B$ is very close to $\Omega$. Due
to the smoothness of $\Phi(x, z)$ for $x$ away from $z$, we can
choose $B$ such that $C_1-\mathcal{F}>c_0>0$. Hence
$
 \vert\tau^{-1}I(v,v)\vert\geq C >0, \ \ \tau\gg1.
$ 
Similarly, considering $v$ to be the $s$-part of the CGOs with linear and logarithmic phases, then Theorem \ref{theorem1} follows from \eqref{gen_ss} and Lemma \ref{ne2}.
\subsubsection{Proof of Theorem \ref{theorem2}}
Recall that, $u_p$ and $u_s$ are the $p$-part and $s$-part of the CGOs for both the linear and logarithmic phases constructed in Proposition \ref{CGO_lin}.
Also, the total field $u$ is $u =u_p+u_s$.
From Lemma \ref{end_lem} and the estimate \eqref{sp_linear}, we have
\[
 \begin{split}
  \frac{I(u_s,u_p)}{\|\nabla u_s\|_{L^2(D)}^{2}}
& \geq (C- \mathcal{F} -{\tilde C}{\epsilon}) + (C -\mathcal{F} - C_2)\frac{\|u_s\|_{L^2(D)}^{2}}{\|\nabla u_s\|_{L^2(D)}^{2}} - C_1\frac{\|\nabla u^s\|_{L^q(D)}^{2}}{\|\nabla u_s\|_{L^2(D)}^{2}} -\frac{\tilde C}{\epsilon} \frac{\|\nabla u\|_{L^2(D)}^{2}}{\|\nabla u_s\|_{L^2(D)}^{2}} \\
&  \geq (C- \mathcal{F} -{\tilde C}{\epsilon}) - c_1\tau^{-2} - c_2\tau^{(1-\frac{2}{q})} - \frac{\tilde c_3}{\epsilon} \tau^{-2}. 
 \end{split}
\]
We can choose $\epsilon > 0$ and $\mathcal{F}$ such that $(C- \mathcal{F} -{\tilde C}{\epsilon}) > c_0 > 0,$ where $c_0$ is another constant. Therefore
$
 \frac{I(u_s,u_p)}{\|\nabla u_s\|_{L^2(D)}^{2}} \geq c_0, \ \text{for} \ \tau\gg1
$
and together with $\|\nabla u_s\|_{L^2(D)}^{2} \geq O(\tau^3),$ we obtain our required estimate
$
 |\tau^{-3}I(u_s,u_p)|>c_0>0, \ \ \tau\gg1.
$
\begin{remark}\label{rem}
 In the case of a penetrable obstacle $D$, the interface $\partial D$ is characterized by the condition $2\mu_D+3\lambda_D\geq 0$
and $\mu_D>0$ (or $2\mu_D+3\lambda_D\leq 0$ and $\mu_D<0$). As we said in the introduction, these conditions are only needed in the vicinity of 
$\partial D,$ since the CGOs and their derivatives are exponentially decaying locally in the interior of $D,$ or its image by the local transformations \eqref{change_coordi} or \eqref{change_coordi_log}, see sections 6 and 7. 
 If $\mu_D=0$ in $D$ or near $\partial D,$ then we cannot conclude by using $s$-parts of the farfield patterns corresponding to $s$ incident waves to reconstruct $D.$
The reason is that the corresponding indicator function is lower bounded by $c\int_D|\nabla \cdot v|^2$ + lot, see \eqref{mou},
where `lot' stands for lower order terms. Since $v$ is $s$ incident type then $\nabla\cdot v=0$. Similarly, we cannot conclude by using $p$-parts of the farfield patterns corresponding to $p$ incident waves for 
the reconstruction as $\nabla\cdot v=-\kappa_{p}^{2}V$ and hence its $L^2(D)$ norm is absorbed by the term 'lot'. The same conclusion applies for the mixed cases. Hence the particular case $\mu_D=0$ and $\lambda_D\neq 0$
near $\partial D$ is not covered by our results. However, this case can be covered if we use the full farfield pattern instead of its $p$ or $s$ parts.
\end{remark}
\subsection{The impenetrable obstacle case}
\subsubsection{Key inequalities}
It is shown in Section 2 that the indicator functions have the general form:
\begin{equation}\label{General form}
 I(v, w):=\int_{\partial D}[u^s(w)\cdot (\overline{\sigma(v)}\cdot \nu)-\overline{v}\cdot(\sigma(u^s(w))\cdot \nu)]ds(x)
\end{equation}
 where $u^s(w)$ is the scattered field associated to the incident field $w$.
Using integration by parts and the boundary conditions, we can write
\[
\begin{split}
 &\int_{\partial D}u^s(w)\cdot (\overline{\sigma(v)}\cdot \nu)ds(x) \\
 & = -\int_{\partial D}u^s(w)\cdot (\overline{\sigma(u^s(v))}\cdot \nu)ds(x) \\
 &=-\int_{\partial \Omega}u^s(w)\cdot (\overline{\sigma(u^s(v))}\cdot \nu)ds(x)+\int_{\Omega \setminus{\overline{D}}} \overline{\sigma(u^s(v))}\cdot \nabla u^s(w) dx -
\kappa^2 \int_{\Omega \setminus\overline{D}} u^s(w)\cdot \overline{u^s(v)}dx \\
\end{split}
\]
and
$$
\int_{\partial D}\overline{v}\cdot(\sigma(u^s(w))\cdot \nu)ds(x)=-\int_{\partial D}\overline{v}\cdot(\sigma(w)\cdot \nu)ds(x)=
-\int_D \sigma(\overline{v})\cdot \nabla w dx +\kappa^2 \int_D w\cdot \overline{v}dx.
$$

Hence \eqref{General form} becomes:
\begin{equation}\label{Identity}
\begin{split}
 I(v, w) =
&-\int_{\partial \Omega}u^s(w)\cdot (\overline{\sigma(u^s(v))}\cdot \nu)ds(x)+\int_{\Omega \setminus\overline{D}} \overline{\sigma(u^s(v))}\cdot \nabla u^s(w) dx 
 -\kappa^2 \int_{\Omega \setminus\overline{D}} u^s(w)\cdot \overline{u^s(v)}dx \\
& +\int_D \overline{\sigma(v)}\cdot \nabla w dx -\kappa^2 \int_D w\cdot \overline{v}dx.\\
\end{split}
\end{equation}
\textbf{The inequalities for $I_{s, s}$ and $I_{p,p}$}:
In this case, we take $v=w$. Hence, we have
$$
I(v,v)\geq -\int_{\partial \Omega}u^s(v)\cdot (\overline{\sigma(u^s(v))}\cdot \nu)ds(x)+\int_D \overline{\sigma(v)}\cdot \nabla v dx-
\kappa^2 \int_{\Omega \setminus\overline{D}} \vert u^s(v)\vert^2dx -\kappa^2 \int_D \vert v\vert^2dx.
$$
By the ellipticity condition of the elasticity tensor and the Korn inequality, we have
$$
\int_D \overline{\sigma(v)}\cdot \nabla v dx =
\int_D \overline{\varrho(\epsilon(v))}\epsilon(v)dx\geq c_1 \int_D \overline{\epsilon(v)}\cdot \epsilon(v)dx\geq \frac{c_1}{C_K}\Vert \nabla v\Vert^2_{L^2(D)}-
c_1\Vert v\Vert^2_{L^2(D)}
$$
where $C_K$ is the Korn constant and $\varrho$ is the elasticity tensor.
Hence
\begin{equation}\label{Omega-v}
I(v,v)\geq -\int_{\partial \Omega}u^s(v)\cdot (\overline{\sigma(u^s(v))}\cdot \nu)ds(x)+c_2\Vert \nabla v\Vert^2_{L^2(D)}-
\kappa^2 \int_{\Omega \setminus\overline{D}} \vert u^s(v)\vert^2dx -(\kappa^2+c_3) \int_D \vert v\vert^2dx.
\end{equation}
\textbf{The inequalities for $I_{s, p}$ and $I_{p,s}$}:
In this case, we take $v\neq w$. We use then the form:
$
I(v, w)= -I(v,v) + I(v,U)
$
where $U := v+w.$ 
As in the penetrable case, using the well posedness of the forward scattering problem and the trace theorem, we show that
$
 |I(v,U)| \leq C\left(\frac{1}{\epsilon}\|\nabla U\|_{L^2(D)}^{2} + \epsilon\|\nabla v\|_{L^2(D)}^{2}\right) \ \text{for} \ 0<\epsilon\ll1.
$
Combining this estimate with \eqref{Omega-v}, we obtain
\begin{equation}\label{imps}
 -I(v,w) \geq I(v,v) - C\left(\frac{1}{\epsilon}\|\nabla U\|_{L^2(D)}^{2} + \epsilon\|\nabla v\|_{L^2(D)}^{2}\right).
\end{equation}
In the next subsection, we derive the lower bound for the indicator function $I(v,v)$ in terms of the CGO solutions only.
\subsubsection{Estimating the dominant terms}
\textbf{Step 1} \\
In this step, we prove that $ \int_{\partial \Omega}u^s(v)\cdot (\overline{\sigma(u^s(v))}\cdot \nu)ds(x)$ is dominated by $\Vert v\Vert_{H^1(D)}^2$.
To do that, we represent $u^s$ as a single layer potential
\begin{equation}\label{integral-representation}
u^s:=\int_{\partial D}\Phi(x, z)f(z)ds(z) := (Sf)(x),
\end{equation} where $f$
satisfies
\begin{equation}\label{integral-equation}
(-\frac{1}{2}I +K^*)f=\sigma_n(u^s)=-\sigma_n(v)
\end{equation}
and $K^*$ is the adjoint of the double layer potential $K$ for the
Lam\'e system.

The following lemma shows that the integral equation \eqref{integral-equation} is solvable in the
spaces $H^{-s}(\partial D)$, for $s \in [-1,1]$ if we assume that
$\kappa^2$ is not an eigenvalue of Dirichlet-Lam\'e operator stated
on $D$.

\begin{lemma}\label{esti_f}
Assume that $\kappa^2$ is not a Dirichlet eigenvalue of the Lam{\'e} operator in $D$.
Then the operator
\[
-\frac{1}{2} I + K^* \ : \ H^{-t}(\partial D) \longrightarrow
H^{-t}(\partial D), \; 0\leq t \leq 1,
\]
is invertible.
\end{lemma}
Now, we use the estimate
$$
\vert \int_{\partial \Omega}u^s(v)\cdot (\overline{\sigma(u^s(v))}\cdot \nu) ds(x)
\vert   \leq \Vert u^s\Vert_{H^{\frac{1}{2}}(\partial \Omega)} \Vert
\sigma_n(u^s)\cdot \nu\Vert_{H^{-\frac{1}{2}}(\partial \Omega)}\leq
C \Vert u^s\Vert^2_{H^{1}(B\setminus\overline{\Omega})}, \ \text{$C>0,$}\
$$
where $B$ is a measurable set containing $\Omega$.
Using the representation \eqref{integral-representation}, we obtain
$$
\Vert u^s\Vert_{H^{1}(B\setminus\overline{\Omega})}\leq \Vert f
\Vert_{H^{-t}(\partial D)}\Vert
F\Vert_{H^{1}(B\setminus\overline{\Omega})}
$$
for $0\leq t\leq 1,$
where $F(x):=\Vert \Phi(x, \cdot)\Vert_{H^{t}(\partial
D)}$. Remark that $\|F\|_{H^1(B\setminus\overline{\Omega})}$ makes sense since $\partial D$ is away from $B\setminus\overline{\Omega}.$
From \eqref{integral-equation} and Proposition \ref{Trace-Theo}, stated at the end of this section, we deduce that
\begin{equation}\label{payel}
\Vert f
\Vert_{H^{-t}(\partial D)} \leq \Vert
\sigma_n(v)\Vert_{H^{-t}(\partial D)}\leq C\Vert
v\Vert_{H^{-t+\frac{3}{2}}(D)}, \ \text{for}\ \frac{1}{2}\leq t <1.
\end{equation}
Therefore, we have:
\begin{equation}\label{us--v}
\vert \int_{\partial \Omega}u^s(v)\cdot (\overline{\sigma(u^s(v))}\cdot \nu) ds(x)
\vert   \leq C \Vert v\Vert^2_{H^{-t+\frac{3}{2}}(D)},
\end{equation}
for $\frac{1}{2}\leq t <1$.\\
\textbf{Step 2} \\
We prove that $\Vert u^s(v)\Vert_{L^2(\Omega\setminus\overline{D})}$ is dominated by $\Vert v\Vert_{H^1(D)}$.
The single layer operator 
$
 S : H^{-t}(\partial D) \rightarrow H_{loc}^{-t+\frac{3}{2}}(\mathbb{R}^3),
$ $\ 0\leq t\leq 1,$
defined in \eqref{integral-representation}, is bounded, see \cite{Mclean}. 
Hence we obtain
\begin{equation}\label{bd_mc}
 \|u^s(v)\|_{L^2(\Omega\setminus\overline{D})} \leq C \|f\|_{H^{-t}(\partial D)}, \ 0\leq t\leq 1.
\end{equation}
From \eqref{payel}, for $\frac{1}{2}\leq t <1,$ we have
\begin{equation}\label{lt}
\Vert u^s\Vert^2_{L^2(\Omega \setminus\overline{D})}\leq \Vert v\Vert^2_{H^{-t+\frac{3}{2}}(D)}. 
\end{equation}
We choose $\frac{1}{2}<t<1$, then by interpolation and using the Cauchy's $\epsilon$-inequality, we obtain:
\begin{equation}\label{us--espilon}
 \Vert v\Vert^2_{H^{-t+\frac{3}{2}}(D)}\leq \epsilon \Vert \nabla v \Vert^2_{L^2(D)}+\frac{C}{\epsilon}\Vert v \Vert^2_{L^2(D)}  
\end{equation}
with some $C>0$ fixed and every $\epsilon>0$.
Finally, combining (\ref{Omega-v}), (\ref{us--v}), \eqref{lt} and (\ref{us--espilon}), we deduce that
\begin{equation}\label{main-estimate}
 I(v, v) \geq c\Vert \nabla v \Vert^2_{L^2(D)} -C\Vert v \Vert^2_{L^2(D)}, \; \tau >>1
\end{equation}
with positive constants $c$ and $C$ independent on $v, w$ and $\tau$ \footnote{We can also replace in the right hand side of (\ref{main-estimate}) $v$ by $w$.}.
Also for two different types of CGOs $v, w$, combining the estimates \eqref{imps} and \eqref{main-estimate}, we obtain
\[
 -I(v,w) \geq (c-\epsilon)\|\nabla v\|_{L^2(D)}^{2} - C\|v\|_{L^2(D)}^{2} - \frac{c_1}{\epsilon}\|\nabla U\|_{L^2(D)}^{2},
\]
recalling that $U = v+w$ is the total field, where $c,C,c_1$ are positive constants and $\epsilon>0$ is chosen very small so that $c-\epsilon$ will be greater than some positive real number.
Therefore, the inequalities for the indicator functions are as follows:
\begin{eqnarray}
  I(u_p,u_p) &\geq& c\Vert \nabla v_p \Vert^2_{L^2(D)} -C\Vert v_p \Vert^2_{L^2(D)},\\
  -I(u_p,u_s) &\geq& (c-\epsilon)\|\nabla v_p\|_{L^2(D)}^{2} - C\|v_p\|_{L^2(D)}^{2} - \frac{c_1}{\epsilon}\|\nabla v\|_{L^2(D)}^{2},\\
 I(u_s,u_s) &\geq& c\Vert \nabla v_s \Vert^2_{L^2(D)} -C\Vert v_s \Vert^2_{L^2(D)},\\
 -I(u_s,u_p) &\geq& (c-\epsilon)\|\nabla v_s\|_{L^2(D)}^{2} - C\|v_s\|_{L^2(D)}^{2} - \frac{c_1}{\epsilon}\|\nabla v\|_{L^2(D)}^{2},
\end{eqnarray}
where $v$ is the CGO solution (with linear or logarithmic phases) for elasticity and $v_p,v_s$ are its $p$ and $s$-parts respectively.
As in the penetrable obstacle case, using the behavior of the CGOs in terms of $\tau,$ we justify Theorem \ref{theorem1} and Theorem \ref{theorem2} 
for the impenetrable obstacle case.
\begin{proof}{of \textbf{Lemma \ref{esti_f}}}
The free space fundamental solutions of the Lam\'e system in $\mathbb{R}^3$
is given by
$
\Phi_{\kappa}(x, y):=\frac{\kappa^2_s}{4\pi \kappa^2}\frac{e^{i\kappa_s \vert
x-y\vert}}{\vert x-y \vert}I+ \frac{1}{4\pi \kappa^2}\nabla_x \nabla_x^T%
\big[ \frac{e^{i\kappa_s \vert x-y\vert}}{\vert x-y \vert}-\frac{%
e^{i\kappa_p \vert x-y\vert}}{\vert x-y \vert}\big]
$
 for a positive frequency $\kappa$ and 
$
\Phi_{0}(x, y):=-\frac{A}{4\pi}\frac{1}{\vert x-y \vert}I-\frac{B}{4\pi}\nabla_x \nabla_x^T%
\big[ \frac{1}{\vert x-y \vert}\big]
$
where $A:=\frac{1}{2}(\frac{1}{\mu}+\frac{1}{\lambda+2\mu})$, $B:=\frac{1}{2}(\frac{1}{\mu}-\frac{1}{\lambda+2\mu})$ 
and $I$ is the identity matrix.
These functions satisfy the equation $[\mu \Delta  +(\lambda +\mu)\nabla \div +\kappa^2] (\Phi_{\kappa}-\Phi_{0})=-\kappa^2 \Phi_{0}, \; \mbox{ in }
\mathbb{R}^3$. Let $Q$ be a large enough domain containing $D$. Multiplying this last equation by $\Phi_{\kappa}$ and 
integrating by parts in $Q$, we obtain
$
(\Phi_{\kappa}-\Phi_{0})(x, y)=-\kappa^2\int_{\Omega}\Phi_{\kappa}(x, z)\Phi_{0}(y, z)dz + BT(\partial \Omega)
$
where the boundary term is
$
BT(\partial \Omega):=\int_{\partial \Omega}(\sigma_n(\Phi_{\kappa}-\Phi_{0})(x, z)\Phi_{\kappa}(y, z)-
\sigma_n(\Phi_{\kappa})(x, z)(\Phi_{\kappa}-\Phi_{0})(y, z))ds(z).
$
From this representation we deduce that
$$
\vert \nabla(\Phi_{\kappa}-\Phi_{0})(x, y)\vert \leq C \int_{\Omega}\vert x-z\vert^{-2}\vert y-z \vert^{-1}dz +C, \; 
\mbox{ for } x, z \mbox{ in or near } D 
$$
where $C$ is a positive constant depending on the distance of $D$ to $\partial Q$. This means that:
$
\vert \nabla(\Phi_{\kappa}-\Phi_{0})(x, y)\vert \leq C \ln(\vert x-z\vert),\; \mbox{ for } x, z \mbox{ in or near } D. 
$
With this estimate at hand, we deduce that $K^*-K^*_0$ is a smoothing operator hence it is a compact operator from $H^{-t}(\partial D)$ to itself. 
Here we denoted by $K^*_0$ the adjoint of the double layer potential for the zero frequency $\kappa=0$.

In \cite{Go-Ke}, it is shown that the operator $-\frac{1}{2}I+K_0: L^2(\partial D)
\rightarrow L^2(\partial D)$ is invertible. In Theorem 13 of \cite{Men-Mit}, see also \cite{Ma-Mit}, it is shown that this operator is also 
invertible from $ H^1(\partial D)$ to $H^1(\partial D)$. Now, arguing as in \cite{Mi}, 
by interpolation and duality, we conclude that $-\frac{1}{2}I+K^*_0: H^{-t}(\partial D)
\rightarrow H^{-t}(\partial D)$ is invertible, for $t \in [0, 1]$.


We write $-\frac{1}{2}I+K^*=(-\frac{1}{2}I+K^*_0)+(K^*- K^*_0)$. 
From the above analysis, this operator is Fredholm of index zero.
Let us now show that it is injective. Let $f\in
H^{-t}(\partial D)$ such that $(-\frac{1}{2}I+K^*)f=0$. This implies
that $(-\frac{1}{2}I+K_0^*)f=-(K^*- K^*_0)f$. Since $K^*- K^*_0$
goes from $H^{-t}$ to $H^{-t+1}$ and $H^{-t+1} \subset L^2(\partial
D)$ for $t\in [0, 1]$, then $(K^*- K^*_0)f \in L^2(\partial D)$. In
addition, $-\frac{1}{2}I+K_0^*$ is invertible from $L^2(\partial D)$
to $L^2(\partial D)$, then we deduce that $f\in L^2(\partial D)$.
Now, by the standard argument, using jumps of the adjoint of the
double layer potential with $L^2(\partial D)$ densities, and the
assumption on $\kappa^2$, we deduce that $f=0$.
\end{proof}

We need the following trace property of the CGO solutions $v$ in the Sobolev spaces of negative orders.
\begin{proposition}\label{Trace-Theo} Let $D$ be a Lipschitz domain and $v$ be a solution of the Helmholtz equation in a set containing D. Then,
there exists a positive constant $C$ independent of $v$ such that
$$
\|\sigma_n(v)\|_{H^{-t}(\partial D)}\leq C
\|v\|_{H^{-t+\frac{3}{2}}(D)}, \; \frac{1}{2} \leq t < 1.
$$
\end{proposition}
\begin{proof} The proof goes in exactly the same way as in the Helmholtz case, see \cite{S-Y}. 
So, we omit to repeat those arguments.
\end{proof}
\begin{section}{\textbf{Estimating the CGOs}}
For our analysis of the inverse problem we needed to estimate the remainder terms of the CGOs.  In the case of linear phase, $\varphi$ is defined as $\varphi := -x\cdot\rho$ and $\psi := -x\cdot\rho^{\bot}$,
 where $\rho, \rho^{\bot}\in \mathbb{S}^2 :=\{x\in \mathbb{R}^3 ; \vert x\vert=1\}$ and $\rho\cdot\rho^{\bot} = 0.$ 
Following the usual notations, see \cite{Ikeha, N-Y, S-Y}, 
we define $K := \partial D \cap\{x\in\mathbb{R}^3 ; x\cdot\rho = h_D(\rho)\}.$ For any $\alpha\in K$, define $B(\alpha,\delta) :=\{x\in\mathbb{R}^3 ; |x-\alpha|<\delta\} \ (\delta>0).$
Then, $K\subset \bigcup_{\alpha\in K}B(\alpha,\delta).$ Since $K$ is compact, there exist $\alpha_1, \cdots , \alpha_N \in K$ such that
$K \subset B(\alpha_1,\delta)\cup \cdots B(\alpha_N, \delta)$. Then
we define
$
 D_{j, \delta}:=D\cap B(\alpha_j, \delta), D_{\delta}:=\bigcup_{j=1}^N D_{j,\delta}.
$
Note that
$
\int_{D\setminus D_{\delta}} e^{-p\tau (h_D(\rho) - x\cdot \rho)}dx = O(e^{-pc\tau}) \ (\tau \to \infty).
$
Let $\alpha_j \in K$. By a rotation and a translation, we may assume that $\alpha_j=0$ and the vector $\alpha_j -
x_0=-x_0$ is parallel to $e_3=(0,0,1)$. Then, we consider a change of
coordinates near $\alpha_j $:
\begin{equation}\label{change_coordi}
 y' = x' , \ y_3 = h_{D}(\rho) - x\cdot \rho,
\end{equation}
where $x'=(x_1,x_2),y'=(y_1,y_2),x=(x',x_3), y=(y', y_3)$. We denote the parametrization of $\partial D,$ after transforming $D$ by \eqref{change_coordi},
 near $\alpha_j$ by $l_j(y')$. Note that we keep the same notation $D$ after the transformation.
\begin{lemma}\label{Remainder_term}
(Analysis of the remainder term for the linear phase.) \\
For $1\leq q \leq 2,$ we have the following estimates for $\tau\gg1$
\begin{enumerate}
 \item 
$
 \int_D e^{-q\tau(h_D(\rho)-x\cdot\rho)}\vert r\vert^qdx  \leq C
\begin{cases} 
  \tau^{-2q-\frac{2-q}{2}}o(1) - \tau^{-2q-\frac{2-q}{2}} e^{-q\tau\delta}, \text{for $r$ in \eqref{keti}.}\ \\
  \tau^{-3q-\frac{2-q}{2}}o(1) - \tau^{-3q-\frac{2-q}{2}} e^{-q\tau\delta}, \text{for $r$ in \eqref{monro} or \eqref{meri}.}\
\end{cases}
$
\item
$
 \int_D e^{-q\tau(h_D(\rho)-x\cdot\rho)}\vert\nabla r\vert^qdx \leq C
\begin{cases} 
 \tau^{-q-\frac{2-q}{2}}o(1) - \tau^{-q-\frac{2-q}{2}} e^{-q\tau\delta}, \text{for $r$ in \eqref{keti}.}\ \\
 \tau^{-2q-\frac{2-q}{2}}o(1) - \tau^{-2q-\frac{2-q}{2}} e^{-q\tau\delta}, \text{for $r$ in \eqref{monro} or \eqref{meri}.}\
\end{cases}
$
\item
$
 \int_D e^{-q\tau(h_D(\rho)-x\cdot\rho)}\vert\nabla\nabla r\vert^qdx \leq C
\begin{cases} 
 \tau^{-\frac{2-q}{2}}o(1) - \tau^{-\frac{2-q}{2}} e^{-q\tau\delta}, \text{for $r$ in \eqref{keti}.}\ \\
 \tau^{-q-\frac{2-q}{2}}o(1) - \tau^{-q-\frac{2-q}{2}} e^{-q\tau\delta}, \text{for $r$ in \eqref{monro} or \eqref{meri}.}\
\end{cases}
$
\item
$
 \int_D e^{-q\tau(h_D(\rho)-x\cdot\rho)}\vert\nabla(\Delta r)\vert^qdx \leq C \tau^{-\frac{2-q}{2}}o(1) - \tau^{-\frac{2-q}{2}} e^{-q\tau\delta}, \text{for $r$ in \eqref{monro} or \eqref{meri}.}
$
\end{enumerate}
The above estimates are also valid if we replace $r$ by $s$, recalling that $R :=\left
     (\begin{array}{c} 
               r \\ 
               s
     \end{array}
      \right)$ 
in \eqref{monro} or \eqref{meri}.
  
 \end{lemma}
\begin{proof}
\begin{enumerate}
 \item
For $1\leq q<2,$ we have
$
 \int_D e^{-q\tau(h_D(\rho)-x\cdot\rho)}\vert r\vert^qdx = \left(\int_{D_{\delta}} + \int_{D\setminus \overline{D_{\delta}}}\right)e^{-q\tau(h_D(\rho)-x\cdot\rho)}\vert r\vert^qdx.
$\\
Applying the H{\"o}lder inequality, we obtain
\[
\begin{split}
& \int_{D_{\delta}}e^{-q\tau(h_D(\rho)-x\cdot\rho)}\vert r\vert^qdx \\
& \leq \left(\int_{D_{\delta}}e^{-pq\tau(h_D(\rho)-x\cdot\rho)}\right)^{\frac{1}{p}} \left(\int_{D_{\delta}}\vert r\vert^{p'q}\right)^{\frac{1}{p'}}, \
 \text{where} \ \frac{1}{p} + \frac{1}{p'} = 1 \ \text{and consider} \ p'q = 2, \ \text{then}\\
& \leq C 
\begin{cases}
\tau^{-2q}\left[\sum_{j = 1}^{N}\iint_{|y'|<\delta}dy'\int_{l_j(y')}^{\delta}e^{-q\tau y_3}dy_3\right]^{\frac{1}{p}}, \text{since}\ \|r\|_{L^2(\Omega)}\leq C\tau^{-2}, \text{see}\ \eqref{es_r} \\
\tau^{-3q}\left[\sum_{j = 1}^{N}\iint_{|y'|<\delta}dy'\int_{l_j(y')}^{\delta}e^{-q\tau y_3}dy_3\right]^{\frac{1}{p}},  \text{since}\ \|r\|_{L^2(\Omega)}\leq C\tau^{-3}, \text{see}\ \eqref{try}
\end{cases}\\
& = C 
\begin{cases}
\tau^{-2q}\left[\tau^{-1}\sum_{j = 1}^{N}\iint_{|y'|<\delta} e^{-pq\tau l_j(y')} - \tau^{-1} e^{-pq\tau\delta} \right]^{\frac{1}{p}}, \\
\tau^{-3q}\left[\tau^{-1}\sum_{j = 1}^{N}\iint_{|y'|<\delta} e^{-pq\tau l_j(y')} - \tau^{-1} e^{-pq\tau\delta} \right]^{\frac{1}{p}}, 
\end{cases} \\
& \leq C
\begin{cases}
 \tau^{-2q-\frac{2-q}{2}}o(1) - \tau^{-2q-\frac{2-q}{2}} e^{-q\tau\delta},  \\
 \tau^{-3q-\frac{2-q}{2}}o(1) - \tau^{-3q-\frac{2-q}{2}} e^{-q\tau\delta}, 
\end{cases}
\end{split}
\]
since $\sum_{j = 1}^{N}\iint_{|y'|<\delta} e^{-pq\tau l_j(y')} = o(1) \ (\tau\rightarrow\infty)$ by the Lebesgue dominated convergence theorem.  \\
For $q=2,$ we have 
\[
\begin{split}
\int_{D_{\delta}}e^{-2\tau(h_D(\rho)-x\cdot\rho)}\vert r\vert^2dx 
& \leq C \|e^{-2\tau(h_D(\rho)-x\cdot\rho)}\|_{L^{\infty}(D_{\delta})}\|r\|_{L^2(D_{\delta})}^{2} \\
& \leq C
\begin{cases}
 \tau^{-4}, \ \text{since}\ \|r\|_{L^2(\Omega)}\leq C\tau^{-2}, \text{see}\ \eqref{es_r} \\
 \tau^{-6}, \ \text{since}\ \|r\|_{L^2(\Omega)}\leq C\tau^{-3}, \text{see}\ \eqref{try}
\end{cases} \\
\end{split}
\]
 \item
Similarly, from the estimates \eqref{winkl} and \eqref{try} with $r=1$, we obtain the required estimate.
\item
Here also, we need the estimates \eqref{winkl} and \eqref{try} with $r=2$.
\item
For this case, we use the property $ \|\nabla(\Delta r)\|_{L^2(D_{\delta})} \leq C$, see \eqref{korc}, to obtain the estimate.
\end{enumerate}
\end{proof}
In the case of logarithmic phase, i.e, $\varphi (x) = \log|x-x_0|$,
and introduce the change of coordinates corresponding to \eqref{change_coordi}:
\begin{equation}\label{change_coordi_log}
 y' = x' , \ y_3 =\log |x-x_0| - d_D(x_0),
\end{equation}
where $x'=(x_1,x_2),y'=(y_1,y_2),x=(x',x_3), y=(y', y_3)$ and 
the parametrization of $\partial D$, after transforming $D$ by \eqref{change_coordi_log},
 near $\alpha_j \in K$ is denoted by $L_j(y')$. Then, we have a similar analysis for the remainder term as in the case of linear phase.
\begin{lemma}\label{Remainder_term_log}
(Analysis of the remainder term for the logarithmic phase.) \\
For $1\leq q \leq 2,$ we have the following estimates for $\tau\gg1$
\begin{enumerate}
 \item 
$
 \int_D e^{-q\tau(\log|x-x_0|-d_D(x_0))}\vert r\vert^qdx \leq C 
\begin{cases} 
  \tau^{-2q-\frac{2-q}{2}}o(1) - \tau^{-2q-\frac{2-q}{2}} e^{-q\tau\delta}, \text{for $r$ in \eqref{keti}.}\ \\
  \tau^{-3q-\frac{2-q}{2}}o(1) - \tau^{-3q-\frac{2-q}{2}} e^{-q\tau\delta}, \text{for $r$ in \eqref{monro} or \eqref{meri}.}\
\end{cases}
$
\item
$
 \int_D e^{-q\tau(\log|x-x_0|-d_D(x_0))}\vert\nabla r\vert^qdx \leq C 
\begin{cases} 
 \tau^{-q-\frac{2-q}{2}}o(1) - \tau^{-q-\frac{2-q}{2}} e^{-q\tau\delta}, \text{for $r$ in \eqref{keti}.}\ \\
 \tau^{-2q-\frac{2-q}{2}}o(1) - \tau^{-2q-\frac{2-q}{2}} e^{-q\tau\delta}, \text{for $r$ in \eqref{monro} or \eqref{meri}.}\
\end{cases}
$
\item
$
 \int_D e^{-q\tau(\log|x-x_0|-d_D(x_0))}\vert\nabla\nabla r\vert^qdx \leq C 
\begin{cases} 
 \tau^{-\frac{2-q}{2}}o(1) - \tau^{-\frac{2-q}{2}} e^{-q\tau\delta}, \text{for $r$ in \eqref{keti}.}\ \\
 \tau^{-q-\frac{2-q}{2}}o(1) - \tau^{-q-\frac{2-q}{2}} e^{-q\tau\delta}, \text{for $r$ in \eqref{monro} or \eqref{meri}.}\
\end{cases}
$
\item
$
 \int_D e^{-q\tau(\log|x-x_0|-d_D(x_0))}\vert\nabla(\Delta r)\vert^qdx \leq C 
[\tau^{-\frac{2-q}{2}}o(1) - \tau^{-\frac{2-q}{2}} e^{-q\tau\delta}], \text{for $r$ in \eqref{monro} or \eqref{meri}.}\
$
\end{enumerate}
The above estimates are also valid if we replace $r$ by $s$, for $r$ in \eqref{monro} or \eqref{meri}.  
 \end{lemma}
The proof is similar to the one of Lemma \ref{Remainder_term} since the estimates are due to the ones provided in proposition \ref{ll} 
and \ref{CGO_lin}, see \eqref{es_r}, \eqref{try} and \eqref{korc}.
\subsection{Estimating the CGOs for $I_{pp}$ and $I_{ss}$}
\ ~ \
\vspace{1mm}
\begin{lemma}\label{ne1}
 Let $v$ be the $p$-part or $s$-part of the CGOs of the linear phase introduced in Proposition \ref{ll}. For $1\leq q \leq 2$, we have the following estimates
for $\tau\gg1$:
\begin{enumerate}
 \item 
$
 \int_D|v|^2dx \geq C\tau\sum_{j=1}^N \iint_{|y'|<\delta }e^{-2\tau l_j(y')}dy' - \tau^{-2}o(1),
$
\item
$
 \int_D|v|^qdx \leq C\tau^{q-1}\sum_{j=1}^N \iint_{|y'|<\delta }e^{-q\tau l_j(y')}dy' + \tau^{-q-\frac{2-q}{2}}o(1),
$
\item
$
\int_D|\nabla v|^2dx \geq C\tau^3\sum_{j=1}^N \iint_{|y'|<\delta }e^{-2\tau l_j(y')}dy' - o(1), 
$
\item
$
 \int_D|\nabla v|^qdx \leq C\tau^{2q-1}\sum_{j=1}^N \iint_{|y'|<\delta }e^{-q\tau l_j(y')}dy' + \tau^{-\frac{2-q}{2}}o(1),
$
where $C$ is a positive constant.
\end{enumerate}
The above estimates are all valid also for logarithmic phase. For this case $l_j$ should be replaced by $L_j.$
\end{lemma}
\begin{proof}
We show the proof for the points (1.) and (2.). Similar arguments can be used to justify (3.) and (4.).
\begin{SCase}(Estimating the CGOs for $I_{pp}$ case.)
\end{SCase}
Recall that the CGOs of Proposition \ref{ll} are of the form
$
 V(x,\tau, t) = e^{-\tau(\varphi+i\psi)}(a_0 + \tau^{-1}a_1 + r),
$
where $a_0, a_1$ are smooth on $\overline{\Omega},$ $a_0 > 0$ and the remainder term $r\in H_{scl}^{2}(\Omega)$ has the estimate
$
  \Vert r\Vert_{H_{scl}^{2}(\Omega)} \leq C\tau^{-2} 
$
and then by interpolation
$
  \Vert r\Vert_{H^{s}(\Omega)} \leq C\tau^{-(2-s)}, \ 0 \leq s \leq 2. 
$
The $p$-part of the CGOs for elasticity is of the form
$
u_p = \nabla V
 =e^{-\tau(\varphi + i\psi)}[-\tau(\nabla\varphi + i\nabla\psi)(a_0 + \tau^{-1}a_1 + r) + \nabla(a_0 + \tau^{-1}a_1 + r)]
$
and then for $j=1,2,3$, we have
\[
\begin{split}
 \frac{\partial u_p}{\partial x_j} 
& = -\tau(\frac{\partial\varphi}{\partial x_j} + i\frac{\partial\psi}{\partial x_j})u_p 
 + e^{-\tau(\varphi + i\psi)}[-\tau(\frac{\partial}{\partial x_j}(\nabla\varphi + i\nabla\psi))(a_0 + \tau^{-1}a_1 + r) \\
& - \tau(\nabla\varphi + i\nabla\psi)(\frac{\partial}{\partial x_j}(a_0 + \tau^{-1}a_1 + r)) + \frac{\partial}{\partial x_j}(\nabla(a_0 + \tau^{-1}a_1 + r))]. \\
\end{split}
\]
\begin{enumerate}
 \item 
Applying Cauchy's $\epsilon$-inequality with $\epsilon = \frac{1}{2}$, we have
\begin{equation}\label{logpp}
 \begin{split}
  \int_D|u_p|^2
& \geq \frac{1}{2}\int_De^{-2\tau\varphi}\{\tau^2|\nabla\varphi+i\nabla\psi|^2|a_0|^2- \\
&|[(-\tau)(\nabla\varphi+i\nabla\psi)(\tau^{-1}a_1+r)+\nabla(a_0+\tau^{-1}a_1+r)]|^2\}\\
& \geq \int_D|a_0|^2\tau^2e^{-2\tau\varphi}-\int_De^{-2\tau\varphi}[\tau^2(\tau^{-2}|a_1|^2+|r|^2)+|\nabla a_0|^2+\tau^{-2}|\nabla a_1|^2+|\nabla r|^2]\\
& \geq \tau^2\left[c_1\int_De^{-2\tau\varphi}-c_2\int_De^{-2\tau\varphi}\{\tau^{-2}+|r|^2+\tau^{-4}+\tau^{-2}|\nabla r|^2\}\right].
 \end{split}
\end{equation}
\item[(i)]
$I_{pp}$ case using the logarithmic phase:\\
Replacing $\varphi(x)$ by $\log|x-x_0|-d_D(x_0)$ in \eqref{logpp} and using the Lemma \ref{Remainder_term_log}
 we have for $\tau\gg1,$
\begin{equation}\label{logpp1}
\begin{split}
 \int_D|u_p|^2 \geq
& \tau^2\left[\{c_1-c_2\tau^{-2}-c_3\tau^{-4}\}\int_De^{-2\tau(\log|x-x_0|-d_D(x_0))}dx\right] \\
& -c[\tau^{-2}o(1)-\tau^{-2}e^{-2\tau\delta}].
\end{split}
\end{equation}
  Therefore for $\tau\gg1$, combining the following estimate
\begin{equation}\label{cincro}
\begin{split}
&\int_De^{-2\tau(\log|x-x_0|-d_D(x_0))}dx \\
&\geq\int_{D_\delta}e^{-2\tau(\log|x-x_0|-d_D(x_0))}dx 
 \geq C \sum_{j=1}^N \iint_{|y'|<\delta }dy' \int_{L_j(y')}^{\delta } e^{-2\tau y_3} dy_3 \\
& \geq C\tau^{-1}\sum_{j=1}^N \iint_{|y'|<\delta }e^{-2\tau L_j(y')}dy' - \frac{C}{2}\tau^{-1}e^{-2\delta \tau}
\end{split}
\end{equation}
together with \eqref{logpp1}, we obtain
\begin{equation}\label{pota}
 \int_D|u_p|^2\geq
 C\tau\sum_{j=1}^N \iint_{|y'|<\delta }e^{-2\tau L_j(y')}dy'- \tau^{-2}o(1) + \text{exponentially decaying term}.
\end{equation}
Observe that in \eqref{pota} the first term is the dominating term as $\sum_{j=1}^N \iint_{|y'|<\delta }e^{-2\tau L_j(y')}dy' \geq C\tau^{-2},$
see \eqref{low_v}.
\item[(ii)]
$I_{pp}$ case using the linear phase:\\
In the linear case, the $p$-part of the CGO solution is explicitly given
\[
 u_p = (\tau\rho + i\sqrt{\tau^2+\kappa_{p}^{2}}\rho^{\perp})e^{\tau(x\cdot\rho - t)+ i\sqrt{\tau^2+\kappa_{p}^{2}}x\cdot\rho^{\perp}}.
\]
The estimate for the lower bound follows as in (i). In this case $a_0,a_1$ and the remainder terms are not appearing, so
the computations are easier. Then replacing $L_j$ by $l_j$
we obtain the required form.
\item
We have for $1\leq q \leq 2,$
\begin{equation}\label{lop1}
   \int_D|u_p|^q
 \leq c_1\int_De^{-q\tau\varphi}\left[\tau^q(1+\tau^{-q}+|r|^q)+1+\tau^{-q}+|\nabla r|^q\right].
\end{equation}
\item[(i)]
$I_{pp}$ case using the logarithmic phase:\\
For $\tau\gg1,$ replacing $\varphi(x)$ by $\log|x-x_0|-d_D(x_0)$ in \eqref{lop1} and using Lemma \ref{Remainder_term_log} 
 we have 
\begin{equation}\label{com}
\begin{split}
 \int_D|u_p|^q \leq 
& c_1\tau^q\left[\int_De^{-q\tau(\log|x-x_0|-d_D(x_0))}(1+\tau^{-q}+\tau^{-2q})\right]\\
& + c[\tau^{-q-\frac{2-q}{2}}o(1)-\tau^{-q-\frac{2-q}{2}}e^{-q\tau\delta}].
\end{split}
\end{equation}
Combining the estimate
\begin{equation}\label{lll1}
\begin{split}
\int_De^{-q\tau(\log|x-x_0|-d_D(x_0))}dx  
& =\left(\int_{D_{\delta} } + \int_{D \setminus D_{\delta }}\right)e^{-q\tau(\log|x-x_0|-d_D(x_0))}dx \\
& \leq C \sum_{j=1}^N \iint_{|y'|<\delta }dy' \int_{L_j(y')}^{\delta } e^{-q\tau y_3} dy_3 + O(e^{-qc\tau}) \\
&\leq C\tau^{-1}\sum_{j=1}^N \iint_{|y'|<\delta }e^{-q\tau L_j(y')}dy' -\frac{C}{q}\tau^{-1} e^{-q\delta \tau} + Ce^{-qc\tau}
\end{split}
\end{equation}
with \eqref{com}, for $\tau\gg1,$ we obtain 
\[
 \int_D|u_p|^q \leq
C\tau^{q-1}\sum_{j=1}^N \iint_{|y'|<\delta }e^{-q\tau L_j(y')}dy' + \tau^{-q-\frac{2-q}{2}}o(1) + \text{exponentially decaying term}.
\]
\item[(ii)]
$I_{pp}$ case using the linear phase:\\
The procedure is similar to the logarithmic case.
\end{enumerate}
\begin{SCase}(Estimating the CGOs for $I_{ss}$ case.)
  \end{SCase}
Let now $v$ be the $s$-part of the CGOs for elasticity. 
Recall that it is given by
$\curl U$ where $U$ is the CGO solution of the vector Helmholtz equation $\Delta + \kappa_{s}^{2}$ of the form
$
U(x,\tau, t) = e^{-\tau(\varphi+i\psi)}(a_0 + \tau^{-1}a_1 + r) (\beta_1, ~\beta_2, ~\beta_3)^\top,
$
 where $a_0, a_1, r$ depend on $\kappa_s$ and $a_0, a_1$ are smooth on $\overline{\Omega}$ and $r$ satisfies \eqref{winkl}. 
Therefore 
$
 u_s = \curl U = (X_{2,3},X_{3,1},X_{1,2}),
$
where
\[
\begin{split}
 X_{l,m} :=
& e^{-\tau(\varphi+i\psi)}\{(\beta_m\frac{\partial a_0}{\partial x_l} - \beta_l\frac{\partial a_0}{\partial x_m}) + \tau^{-1}(\beta_m\frac{\partial a_1}{\partial x_l} - \beta_l\frac{\partial a_1}{\partial x_m}) + (\beta_m\frac{\partial r}{\partial x_l} - \beta_l\frac{\partial r}{\partial x_m}) \\
& -\tau(a_0 + \tau^{-1}a_1 + r)((\beta_m\frac{\partial\varphi}{\partial x_l} - \beta_l\frac{\partial\varphi}{\partial x_m}) + i(\beta_m\frac{\partial\psi}{\partial x_l} - \beta_l\frac{\partial\psi}{\partial x_m}))\} 
\end{split}
\]
with $(l,m) = (2,3), (3,1)$ and $(1,2).$ 
Also the gradient term can be written as 
$
 \nabla u_s = (\nabla X_{2,3},\nabla X_{3,1},\nabla X_{1,2}),
$
where
\[
\begin{split}
 \frac{\partial X_{l,m}}{\partial x_j} = 
&-\tau (\frac{\partial\varphi}{\partial x_j} + i\frac{\partial\psi}{\partial x_j})X_{l,m}+ e^{-\tau(\varphi+i\psi)}\{(\beta_m\frac{\partial^2a_0}{\partial x_j\partial x_l} - \beta_l\frac{\partial^2a_0}{\partial x_j\partial x_m}) \\
&  + \tau^{-1}(\beta_m\frac{\partial^2a_1}{\partial x_j\partial x_l} - \beta_l\frac{\partial^2a_1}{\partial x_j\partial x_m}) 
+ (\beta_m\frac{\partial^2 r}{\partial x_j\partial x_l} - \beta_l\frac{\partial^2 r}{\partial x_j\partial x_m}) \\
& -\tau(\frac{\partial a_0}{\partial x_j} + \tau^{-1}\frac{\partial a_1}{\partial x_j} + \frac{\partial r}{\partial x_j})((\beta_m\frac{\partial\varphi}{\partial x_l} - \beta_l\frac{\partial\varphi}{\partial x_m}) + i(\beta_m\frac{\partial\psi}{\partial x_l} - \beta_l\frac{\partial\psi}{\partial x_m})) \\
& -\tau(a_0 + \tau^{-1}a_1 + r)((\beta_m\frac{\partial^2\varphi}{\partial x_j\partial x_l} - \beta_l\frac{\partial^2\varphi}{\partial x_j\partial x_m}) + i(\beta_m\frac{\partial^2\psi}{\partial x_j\partial x_l}-\beta_l\frac{\partial^2\psi}{\partial x_j\partial x_m}))
\}.
\end{split}
\]
\begin{enumerate}
 \item 
Arguing as in Case 1, for $\tau\gg1,$ using Cauchy's $\epsilon$-inequality and Lemma \ref{Remainder_term}, the dominant term for the lower bound of $\| v_s\|_{L^2(D)}^{2}$ has the following form
\[
 \int_D|u_s|^2 \geq
 C\tau^2\int_De^{-2\tau\varphi}\sum_{l,m}\left[(\beta_m\frac{\partial\varphi}{\partial x_l} - \beta_l\frac{\partial\varphi}{\partial x_m})^2 + (\beta_m\frac{\partial\psi}{\partial x_l} - \beta_l\frac{\partial\psi}{\partial x_m})^2\right] -\tau^{-2}o(1).
\]
\item[(i)]
$I_{ss}$ case using the logarithmic phase:\\
Using the explicit forms of $\varphi$ and $\psi$, we need the term
\begin{equation}\label{mili}
\sum_{l,m}\left[(\beta_m\frac{\partial\varphi}{\partial x_l} - \beta_l\frac{\partial\varphi}{\partial x_m})^2 + (\beta_m\frac{\partial\psi}{\partial x_l} - \beta_l\frac{\partial\psi}{\partial x_m})^2\right] 
\end{equation}
to be greater than some positive constant in $D$. If the term \eqref{mili} is equal to zero, then both the square terms will be zero. However
$(\beta_m\frac{\partial\varphi}{\partial x_l} - \beta_l\frac{\partial\varphi}{\partial x_m})^2(x) = 0$ implies $x$ lies on the line
$\{x\in\mathbb{R}^3; x-x_0 = \frac{x_1-x_{01}}{\beta_1}(\beta_1, \beta_2, \beta_3)\}$ which we write as
$\{x\in\mathbb{R}^3, x-x_0=\lambda(\beta_1,\beta_2,\beta_3), \ \lambda\in\mathbb{R}\}$. 
So, it is enough to choose $(\beta_1,\beta_2,\beta_3)$ such that $\{x\in\mathbb{R}^3, x-x_0=\lambda(\beta_1,\beta_2,\beta_3), \ \lambda\in\mathbb{R}\}\cap\overline{D} =\emptyset.$
A natural choice is $(\beta_1,\beta_2,\beta_3) :=\omega_0$ where $\omega_0$ is given in \eqref{mcon} in Proposition \ref{CGO_lin} since $\overline{D}\subset\Omega$ and
$\{x\in\mathbb{R}^3, x-x_0=\lambda\omega_0, \ \lambda\in\mathbb{R}\}\cap\partial\Omega=\emptyset.$ 
The remaining part of the proof is the same as in Case 1.
\item[(ii)]
$I_{ss}$ case using the linear phase:\\
In this case, the CGO has a simple form, see Proposition \ref{ll} and \eqref{mili} is equal to \\
 $\sum_{l,m}\left[(\rho_l-\rho_m)^2+(\rho_{l}^{\bot}-\rho_{m}^{\bot})^2\right]$, so it is positive in $D.$
\item
As in Case 1, the leading order term is $\tau^{2q}$ for $1\leq q \leq2$, so using the previous argument we have the result for both the linear and logarithmic phases.
\end{enumerate}
\end{proof}
\begin{lemma}\label{ne2}
Let $v$ be the $p$-part or $s$-part of the CGOs of both 
the linear and logarithmic phases introduced in Section \ref{cgo_PP}. Let also $V$ be the solution of $(\Delta + \kappa_{p}^{2})V=0$
introduced in Section \ref{cgo_PP}.
Then we have the following estimates 
\begin{enumerate}
 \item 
$
 \frac{\Vert\nabla v\Vert_{L^2(D)}^{2}}{\Vert v\Vert_{L^2(D)}^{2}} \geq C\tau^2, \ \ \tau\gg1
$
\item
$
 \frac{\Vert V\Vert_{L^2(D)}^{2}}{\Vert v\Vert_{L^2(D)}^{2}} \leq C\tau^{-2} \ \text{and} \ \Vert v\Vert_{L^2(D)}^{2} \geq C\tau^{-1}, \ \ \tau\gg1
$
\item For $q_0<q<2$, 
$
 \frac{\Vert\nabla v\Vert_{L^q(\Omega)}^{2}}{\Vert v\Vert_{L^2(D)}^{2}}\leq C\tau^{3-\frac{2}{q}}, \ \ \tau\gg1.
$
\end{enumerate}
\end{lemma}
\begin{proof}
Note that
 \begin{equation}\label{low_v}
\begin{split}
\sum_{j=1}^N \iint_{|y'|<\delta} e^{-2\tau l_j(y')} dy' 
&\geq C\sum_{j=1}^N \iint_{|y'|<\delta} e^{-2\tau |y'|} dy' \\
&\geq C\tau^{-2} \sum_{j=1}^N \iint_{|y'|<\tau \delta} e^{-2 |y'|} dy' 
=C\tau^{-2} \ \ (\tau \gg 1),
\end{split}
\end{equation}
since we have $l_j(y')\leq C|y'|$ if $\partial D$ is Lipschitz.
\begin{enumerate}
 \item 
Using Lemma \ref{ne1} and \eqref{low_v}, we obtain
\[
\begin{split}
\frac{\int_D |\nabla v|^2 dx}{\int_D |v|^2dx} &\geq
C\frac{\tau^3 \sum_{j=1}^N \iint_{|y'|<\delta} e^{-2\tau l_j(y')} dy'- o(1) }{\tau \sum_{j=1}^N \iint_{|y'|<\delta} e^{-2\tau l_j(y')} dy' + \tau^{-2}o(1)} \\
&=C\tau^2 \frac{1-\frac{\tau^{-3}o(1)}{\sum_{j=1}^N \iint_{|y'|<\delta} e^{-2\tau l_j(y')} dy'} }{ 1 +\frac{\tau^{-3}o(1)}{\sum_{j=1}^N \iint_{|y'|<\delta} e^{-2\tau l_j(y')} dy'}} 
=C\tau^2 \ \ (\tau \gg 1).
\end{split}
\]
\item
The proof of this part is similar to (1.).
\item
Using the H{\"o}lder inequality with exponent $\tilde q=\frac{2}{q}>1,$ we have
\[
 \sum_{j=1}^N \iint_{|y'|<\delta} e^{-q\tau l_j(y')} dy'\leq C [\sum_{j=1}^N \iint_{|y'|<\delta} e^{-2\tau l_j(y')} dy']^{\frac{q}{2}}.
\]
Therefore
\[
 \begin{split}
 \frac{\Vert\nabla v\Vert_{L^q(\Omega)}^{2}}{\Vert v\Vert_{L^2(D)}^{2}}
& \leq \frac{[C\tau^{2q-1}\sum_{j=1}^N \iint_{|y'|<\delta} e^{-q\tau l_j(y')} dy' + \tau^{-\frac{2-q}{2}}o(1) ]^{\frac{2}{q}}}
{[C\tau\sum_{j=1}^N \iint_{|y'|<\delta} e^{-2\tau l_j(y')} dy' - \tau^{-2}o(1)]} \\
& \leq \frac{C\tau^{(2q-1)\frac{2}{q}}\sum_{j=1}^N \iint_{|y'|<\delta} e^{-2\tau l_j(y')} dy' + \tau^{-\frac{2-q}{q}}o(1)}
{C\tau\sum_{j=1}^N \iint_{|y'|<\delta} e^{-2\tau l_j(y')} dy'- \tau^{-2}o(1)} 
 \leq C\tau^{3-\frac{2}{q}}, \ \ (\tau \gg 1). 
\end{split}
\]
\end{enumerate}
Note that the proof of the logarithmic case is similar, the only change is to replace $l_j$ by $L_j$.
\end{proof}
\end{section}
\subsection{\textbf{Estimating the CGOs for $I_{ps}$ and $I_{sp}$}}
\ ~ \
\vspace{1mm}
\begin{lemma}\label{analysis_lin}
 Consider $u_s$ to be the $s$-part of the CGOs with linear and logarithmic phases, introduced
in Proposition \ref{CGO_lin}. Then we have the following estimates for the linear phase: For $1\leq q \leq 2$, we have
\begin{enumerate}
\item 
$
\int_D |u_s|^q dx \leq C\tau^{2q-1} \sum_{j=1}^N \iint_{|y'|<\delta} e^{-q\tau l_j(y')} dy'  + \tau^{-q-\frac{2-q}{2}}o(1),
$
\item
$
\int_D |u_s|^2 dx \geq C\tau^3 \sum_{j=1}^N \iint_{|y'|<\delta} e^{-2\tau l_j(y')} dy' -\tau^{-2}o(1),
$
\item 
$
\int_D |\nabla u_s|^q dx \leq C\tau^{3q-1} \sum_{j=1}^N \iint_{|y'|<\delta} e^{-q\tau l_j(y')} dy'+ \tau^{-\frac{2-q}{2}}o(1),
$
\item
$
\int_D |\nabla u_s|^2 dx \geq C\tau^5 \sum_{j=1}^N \iint_{|y'|<\delta} e^{-2\tau l_j(y')} dy' - o(1).
$
\end{enumerate}
In the logarithmic case, the all above estimates hold replacing $l_j$ in each estimate by $L_j$.
\end{lemma}
 \begin{proof}
We give the proofs for the points (1) and (2). The proof of (3) and (4) are similar.
\begin{enumerate}
 \item Assume that $\varphi$ be the linear phase, i.e $\varphi(x)=t-x\cdot\rho, t>0$.
Now, for $t=h_D(\rho),$ for the $\tau\gg1,$ the estimate \eqref{upper_q1} in Appendix together with the behavior of the remainder term in Lemma \ref{Remainder_term},
gives
\[
 \begin{split}
  \int_D \vert u_s\vert^q dx 
&\leq C\tau^{2q}\int_D e^{-q\tau(h_D(\rho)-x\cdot\rho)}dx + \tau^{-q-\frac{2-q}{q}}o(1)\\
&\leq C\tau^{2q-1}\sum_{j=1}^N \iint_{|y'|<\delta }e^{-q\tau l_j(y')}dy' + \tau^{-q-\frac{2-q}{q}}o(1)+ \ \text{exponentially decaying terms}.
 \end{split}
\]
Note that the last inequality follows as in \eqref{lll1}.
Similarly, applying Lemma \ref{Remainder_term_log} and the estimate \eqref{upper_q1} in Appendix, we prove the required estimate for the case of logarithmic phase.
\item
\item[(i)](Linear case.)
Consider $\varphi$ to be the linear phase. From the estimate \eqref{sne1} see Appendix and Lemma \ref{Remainder_term}, we have, for $\tau\gg1,$
\begin{equation}\label{mi}
\int_D|u_s|^2dx
 \geq C\tau^4\int_De^{-2\tau(h_D(\rho)-x\cdot\rho)}|a_0|^2 -\tau^{-2}o(1).
\end{equation}
Recall that $a_0= \rho$, then $|a_0|=|\rho|=1.$ Therefore for large $\tau$, from \eqref{mi} and \eqref{cincro} we obtain
\[
 \begin{split}
\int_D|u_s|^2dx
& \geq  C\tau^4\int_De^{-2\tau(h_D(\rho)-x\cdot\rho)}dx -\tau^{-2}o(1) 
\geq C\tau^4\int_{D_{\delta} } e^{-2\tau (h_D(\rho)-x\cdot \rho)} dx-\tau^{-2}o(1) \\
&\geq C\tau^4\sum_{j=1}^N \iint_{|y'|<\delta }dy' \int_{l_j(y')}^{\delta } e^{-2\tau y_3} dy_3 -\tau^{-2}o(1)+ \ \text{exponentially decaying terms} \\
&=C\tau^3\sum_{j=1}^N \iint_{|y'|<\delta }e^{-2\tau l_j(y')}dy' -\tau^{-2}o(1)+ \ \text{exponentially decaying terms}.
 \end{split}
\]
\item[(ii)](Logarithmic case.)
Consider $\varphi(x) = \log|x-x_0|$.
Recall that the support function has the following form 
$
 d_D(x_0) := \inf_{x\in D}\log|x-x_0|. \ (x_0 \in \mathbb{R}^3\setminus\overline{ch(\Omega)}).
$
From \eqref{sne1} in Appendix and Lemma \ref{Remainder_term_log}, we have, for $\tau\gg1,$
\begin{equation}\label{mim}
\int_D|u_s|^2dx
 \geq C\tau^4\int_De^{-2\tau(d_D(x_0)-\log|x-x_0|)}|a_0|^2 -\tau^{-2}o(1).
\end{equation}
 Since $|a_0|>c_0>0$ in $D$, see \eqref{solCR}, therefore \eqref{mim} reduces to
\[
  \begin{split}
\int_D|u_s|^2dx
& \geq C\tau^4\int_De^{-2\tau(d_D(x_0)-\log|x-x_0|)}dx -\tau^{-2}o(1)\\
&\geq C\tau^3\sum_{j=1}^N \iint_{|y'|<\delta }e^{-2\tau L_j(y')}dy'-\tau^{-2}o(1) + \ \text{exponentially decaying terms}.
\end{split}
\]
Note that the last inequality follows from \eqref{cincro}.
\end{enumerate}
\end{proof}
\begin{lemma}\label{linear_u}
Let $u$ be the CGO solution constructed in Proposition \ref{CGO_lin}. If $u$ has a linear phase then we have
\[ 
\Vert \nabla u\Vert_{L^2(D)}^{2} \leq C\tau^3\sum_{j = 1}^{N}\iint_{|y'|<\delta}e^{-2\tau l_j(y')}dy' + \text{exponentially decaying term}.
\]
The estimate is true for the CGOs with logarithmic phase as well replacing $l_j$ by $L_j$.
\end{lemma}
\begin{proof}
 Let us recall that the CGO solution for the elasticity model is of the form $u = \mu_{0}^{-\frac{1}{2}}w + \mu_{0}^{-1}\nabla g,$ where $w$ and $g$ are defined as in Proposition \ref{CGO_lin}.
Considering the behavior of the remainder term from Lemma \ref{Remainder_term} and doing the same analysis as given in the proof of Lemma \ref{analysis_lin}, we deduce that
\[
 \begin{split}
  \|\nabla u\|_{L^2(D)}^{2} 
& \leq C \left[\|\nabla w\|_{L^2(D)}^{2} + \|\nabla\nabla g\|_{L^2(D)}^{2} \right] 
 \leq C\tau^4\int_D e^{-2\tau(h_D(\rho)-x\cdot\rho)} dx \\
& \leq C\tau^3\sum_{j = 1}^{N}\iint_{|y'|<\delta}e^{-2\tau l_j(y')}dy' + \text{exponentially decaying term}.
 \end{split}
\]
\end{proof}
\begin{lemma}\label{end_lem}
Let $u$ be the CGO solution constructed in Proposition \ref{CGO_lin} and $u_s$ be its $s$-part.
 Then we have the following estimates:
\begin{enumerate}
 \item 
$
\frac{\|\nabla u\|_{L^2(D)}^{2}}{\|\nabla u_s\|_{L^2(D)}^{2}} \leq C \tau^{-2} \ \text{and} \
 \frac{\|u_s\|_{L^2(D)}^{2}}{\|\nabla u_s\|_{L^2(D)}^{2}} \leq C \tau^{-2}.
$
\item
For $q_0<q< 2,$
$
 \frac{\|\nabla u_s\|_{L^q(D)}^{2}}{\|\nabla u_s\|_{L^2(D)}^{2}} \leq C \tau^{(1-\frac{2}{q})}, \ (\tau\rightarrow\infty).
$
\end{enumerate}
\end{lemma}
\begin{proof}
 The proof follows from Lemma \ref{analysis_lin} and Lemma \ref{linear_u}.
\end{proof}
\begin{section}{Appendix}
In this section, we are only concerned with the CGOs constructed in Proposition \ref{CGO_lin}.
 Let $\varphi + i\psi$ be the phase function, where $\varphi$ is either linear or logarithmic. Recall that 
the $s$-part of the CGO solution is
\begin{equation}\label{cris}
 u_s = \kappa_{s}^{-2}\mu_{0}^{-\frac{1}{2}}\left[\nabla(\nabla\cdot w) - \Delta w\right]
\end{equation}
where $ w = e^{-\tau(\varphi+i\psi}(a_0 + \tau^{-1}a_1 + \tau^{-2}a_2 + r)$. 
Writing \eqref{cris} explicitly we obtain
$u_s = (u_{s}^{1}, ~u_{s}^{2}, ~u_{s}^{3}) ^\top$
where 
$
 u_{s}^{j} := \kappa_{s}^{-2}\mu_{0}^{-\frac{1}{2}}e^{-\tau(\varphi+i\psi)}\left[ I - II\right]
$
with
 \begin{eqnarray}\label{I}
\begin{split}
 & I := \left[\tau^2(\nabla(\varphi+i\psi))\cdot(a_0 + \tau^{-1}a_1 + \tau^{-2}a_2 + r)-\tau\nabla\cdot(a_0 + \tau^{-1}a_1 + \tau^{-2}a_2 + r)\right]\frac{\partial}{\partial x_j}(\varphi+i\psi) \\
 & + \frac{\partial}{\partial x_j}\left[ -\tau(\nabla(\varphi+i\psi))\cdot(a_0 + \tau^{-1}a_1 + \tau^{-2}a_2 + r) + \nabla\cdot(a_0 + \tau^{-1}a_1 + \tau^{-2}a_2 + r)\right], \ \text{and}\
\end{split}
\\ \label{II}
 \begin{split}
  & II := -\tau(\Delta(\varphi+i\psi))(a_{0}^{j} + \tau^{-1}a_{1}^{j} + \tau^{-2}a_{2}^{j} + r^j) - 2\tau(\nabla(\varphi+i\psi))\cdot\nabla(a_{0}^{j} + \tau^{-1}a_{1}^{j} + \tau^{-2}a_{2}^{j} + r^j) \\
 & + (\Delta a_{0}^{j} + \tau^{-1}\Delta a_{1}^{j} + \tau^{-2}\Delta a_{2}^{j} + \Delta r^j).
 \end{split}
\end{eqnarray}
Here $a_{0}^{j}, a_{1}^{j}, a_{2}^{j}$ and $r^j$ are the $j^{th}$ component of $a_0, a_1, a_2$ and $r$ respectively, $j = 1, 2, 3.$ 
\subsection{Estimating the lower bound for $\Vert u_s\Vert_{L^2(D)}$}
Applying Cauchy's $\epsilon$ inequality with $\epsilon = \frac{1}{2}$, we have
$
 \vert u_{s}^{j}\vert^2 \geq \kappa_{s}^{-4}\mu_{0}^{-1}e^{-2\tau\varphi}( \frac{1}{2} \vert I\vert^2 - \vert II\vert^2).
$
To get the lower estimate for $\vert I\vert^2$, using again Cauchy's $\epsilon$ inequality with $\epsilon = \frac{1}{2}$, we have
\begin{equation}\label{epsilon_1}
 \begin{split}
&\sum_{j = 1}^{3}\int_D e^{-2\tau\varphi} \vert I\vert^2  \geq \\
&\frac{1}{2}\sum_{j = 1}^{3}\int_D e^{-2\tau\varphi}\vert\tau^2(\nabla(\varphi+i\psi)\cdot a_0)\frac{\partial}{\partial x_j}(\varphi+i\psi)\vert^2 -\sum_{j = 1}^{3}\int_D e^{-2\tau\varphi}\{c_1\vert\tau^2(\nabla(\varphi+i\psi)\cdot\tau^{-1} a_1)\frac{\partial}{\partial x_j}(\varphi+i\psi)\vert^2 \\
& + c_2 \vert\tau^2(\nabla(\varphi+i\psi)\cdot\tau^{-2} a_2)\frac{\partial}{\partial x_j}(\varphi+i\psi)\vert^2 + c_3\vert\tau^2(\nabla(\varphi+i\psi)\cdot r)\frac{\partial}{\partial x_j}(\varphi+i\psi)\vert^2\}\\
& - C \sum_{j = 1}^{3}\int_D e^{-2\tau\varphi}\{\vert\tau(\nabla\cdot a_0)\frac{\partial}{\partial x_j}(\varphi+i\psi)\vert^2 + \vert\tau(\nabla\cdot\tau^{-1} a_1)\frac{\partial}{\partial x_j}(\varphi+i\psi)\vert^2 
+ \vert\tau(\nabla\cdot\tau^{-2}a_2)\frac{\partial}{\partial x_j}(\varphi+i\psi)\vert^2\} 
\end{split}
\end{equation}
\begin{equation}
\begin{split}
& + \vert\tau(\nabla\cdot r)\frac{\partial}{\partial x_j}(\varphi+i\psi)\vert^2\} 
 - C \{\vert\frac{\partial}{\partial x_j}(-\tau\nabla(\varphi+i\psi)\cdot a_0)\vert^2 + \vert\frac{\partial}{\partial x_j}(-\tau\nabla(\varphi+i\psi)\cdot\tau^{-1} a_1)\vert^2 \\ 
& + \vert\frac{\partial}{\partial x_j}(-\tau\nabla(\varphi+i\psi)\cdot \tau^{-2}a_2)\vert^2 + \vert\frac{\partial}{\partial x_j}(-\tau\nabla(\varphi+i\psi)\cdot r)\vert^2 \\ 
& + \vert\frac{\partial}{\partial x_j}(\nabla\cdot a_0)\vert^2 + \vert\frac{\partial}{\partial x_j}(\nabla\cdot\tau^{-1} a_1)\vert^2 + \vert\frac{\partial}{\partial x_j}(\nabla\cdot\tau^{-2} a_2)\vert^2 
+ \vert\frac{\partial}{\partial x_j}(\nabla\cdot r)\vert^2\}\\
&\geq \int_De^{-2\tau\varphi}\left[\frac{1}{2}\tau^4|a_0|^2-c_1(1+\tau^2+\tau^{-2}+\tau^{-4})-c_2(\tau^4|r|^2+\tau^2|\nabla r|^2+|\nabla\nabla r|^2)\right].
 \end{split}
\end{equation}
Computations for the term $II:$
\begin{equation}\label{epsilon_4}
\begin{split}
& \sum_{j = 1}^{3}\int_D e^{-2\tau\varphi} \vert II\vert^2 
  \leq C \sum_{j = 1}^{3}\int_D e^{-2\tau\varphi} \{\tau^2\vert(\Delta(\varphi+i\psi))a_{0}^{j}\vert^2 + \vert(\Delta(\varphi+i\psi))a_{1}^{j}\vert^2 + \tau^{-2}\vert(\Delta(\varphi+i\psi))a_{2}^{j}\vert^2 \\
& +\tau^2\vert(\Delta(\varphi+i\psi))r^j\vert^2 + \tau^2\vert\nabla(\varphi+i\psi)\cdot\nabla a_{0}^{j}\vert^2 + \vert\nabla(\varphi+i\psi)\cdot\nabla a_{1}^{j}\vert^2
+ \tau^{-2}\vert\nabla(\varphi+i\psi)\cdot\nabla a_{2}^{j}\vert^2 \\
& + \tau^2\vert\nabla(\varphi+i\psi)\cdot\nabla r^j\vert^2  + \vert\Delta a_{0}^{j}\vert^2 + \tau^{-2}\vert\Delta a_{1}^{j}\vert^2 + \tau^{-4}\vert\Delta a_{2}^{j}\vert^2 + \vert\Delta r^j\vert^2\} \\
& \leq C \int_D e^{-2\tau\varphi}\left[ 1 + \tau^2 + \tau^{-2} + \tau^{-4}\right]dx + C \left[\int_D \tau^2 e^{-2\tau\varphi}\vert r\vert^2 + \int_D \tau^2 e^{-2\tau\varphi}\vert\nabla r\vert^2 + \int_D e^{-2\tau\varphi}\vert\Delta r\vert^2 \right].
\end{split}
\end{equation}
Therefore from \eqref{epsilon_1} and \eqref{epsilon_4} we obtain the lower bound of $\Vert u_s\Vert_{L^2(D)}$ as follows:
\begin{equation}\label{sne1}
\begin{split}
\int_D|u_s|^2dx
& \geq c\tau^4\int_De^{-2\tau\varphi}[|a_0|^2-c_1\{\tau^{-2}+\tau^{-4}+\tau^{-6}+\tau^{-8}\}\\
& -c_2\{\tau^{-2}|r|^2+\tau^{-2}|\nabla r|^2+\tau^{-4}|\nabla\nabla r|^2+|r|^2+\tau^{-2}|\nabla\cdot r|^2 \}].\\
 \end{split}
\end{equation}
 \subsection{Estimating the upper bound for $\Vert u_s\Vert_{L^q(D)}$} 
 For $1\leq q \leq 2,$ we have
the following upper estimate:
  \begin{equation}\label{upper_q1}
  \begin{split}
\int_D \vert u_s\vert^q
& \leq C \int_D e^{-q\tau\varphi}[1 + \tau^q +\tau^{2q}+ \tau^{-q} + \tau^{-2q}] \\
& + \int_D e^{-q\tau\varphi}\left[\tau^q\vert r\vert^q +\tau^{2q}\vert r\vert^{q}+ \tau^q\vert\nabla r\vert^q + \tau^q\vert\nabla\cdot r\vert^q +\vert\nabla\nabla r\vert^q + \vert\Delta r\vert^q\right].
\end{split}
\end{equation}
\subsection{Estimating the lower bound for $\Vert\nabla u_s\Vert_{L^2(D)}$} 
We have $(\nabla u_{s}^{1}, ~\nabla u_{s}^{2}, ~\nabla u_{s}^{3}) ^\top $
where, for $j=1,2,3,$
$
\nabla u_{s}^{j} 
  = \kappa_{s}^{-2}\mu_{0}^{-\frac{1}{2}}\left[\nabla e^{-\tau(\varphi+i\psi)}(I - II) + e^{-\tau(\varphi+i\psi)}(\nabla I-\nabla II)\right] 
  =  \kappa_{s}^{-2}\mu_{0}^{-\frac{1}{2}} e^{-\tau(\varphi+i\psi)}[A + B]
$
with $A := -\tau(\nabla\varphi+i\nabla\psi)(I - II),$ $B := \nabla I - \nabla II$ and $I,II$  are defined as in \eqref{I} and \eqref{II} respectively.
Therefore, 
\begin{equation}\label{grad_2}
 \int_D\vert\nabla u_s\vert^2dx
\geq \kappa_{s}^{-4}\mu_{0}^{-1}\int_D e^{-2\tau\varphi}\sum_{j=1}^{3}[\frac{1}{2}\vert A\vert^2 - \vert B\vert^2].
\end{equation}
To get the estimate for the term $A$, we use \eqref{epsilon_1}, \eqref{epsilon_4} and the following estimate
\begin{equation}\label{grad_3}
 \int_De^{-2\tau\varphi}\sum_{j=1}^{3}\vert A\vert^2 \geq \tau^2\int_De^{-2\tau\varphi}\sum_{j=1}^{3}[\frac{1}{2}\vert I\vert^2 - \vert II\vert^2].
\end{equation}
Again we have 
\begin{equation}\label{grad_4}
\begin{split}
 &\int_D e^{-2\tau\varphi}\sum_{j=1}^{3}\vert B\vert^2dx 
 \leq C \int_D e^{-2\tau\varphi}\sum_{j=1}^{3}\left[\vert\nabla I\vert^2 + \vert\nabla II\vert^2\right] 
 \leq C\{\int_D e^{-2\tau\varphi}\left[ 1+ \tau^2 + \tau^4 + \tau^{-2} + \tau^{-4}\right] \\
& + \int_D e^{-2\tau\varphi}\left[ \tau^4\vert\nabla(\nabla(\varphi+i\psi)\cdot r)\vert^2 + \tau^2 \vert\nabla(\nabla\cdot r)\vert^2 + \tau^2\sum_{j=1}^{3}\vert\nabla(\frac{\partial}{\partial x_j}(\nabla(\varphi + i\psi)\cdot r))\vert^2\right] \\
\end{split}
\end{equation}
\begin{equation*}
\begin{split}
& + \int_D e^{-2\tau\varphi}\sum_{j=1}^{3}\left[\vert \nabla(\frac{\partial}{\partial x_j}(\nabla\cdot r))\vert^2 + \tau^2\vert \nabla(r^j\Delta(\varphi+i\psi))\vert^2 +\tau^2\vert\nabla(\nabla(\varphi+i\psi)\cdot \nabla r^j)\vert^2\right] \\
& + \int_D e^{-2\tau\varphi}\sum_{j=1}^{3} \vert \nabla(\Delta r^j)\vert^2\}.
\end{split}
\end{equation*}
Combining \eqref{grad_2}, \eqref{grad_3} and \eqref{grad_4}, we get the lower bound for the gradient term as follows:
\begin{equation*}
 \begin{split}
  \int_D|\nabla u_s|^2dx
& \geq c\tau^6\int_De^{-2\tau\varphi}[|a_0|^2-c_1\{\tau^{-2}+\tau^{-4}+\tau^{-6}+\tau^{-8}+\tau^{-10}\} \\
&-c_2\{\tau^{-2}|r|^2+\tau^{-2}|\nabla r|^2+\tau^{-4}|\Delta r|^2+|r|^2+\tau^{-2}|\nabla\cdot r|^2+\tau^{-4}\sum_{j=1}^{3}|\frac{\partial}{\partial x_j}(\nabla\cdot r)|^2\}\\
& -c_3\{\tau^{-2}|\nabla(\nabla(\varphi+i\psi)\cdot r)|^2+\tau^{-4}|\nabla(\nabla\cdot r)|^2+\tau^{-4}\sum_{j=1}^{3}|\nabla(\frac{\partial}{\partial x_j}(\nabla(\varphi+i\psi)\cdot r))|^2\\ 
&+\tau^{-6}\sum_{j=1}^{3}|\nabla(\frac{\partial}{\partial x_j}(\nabla\cdot r))|^2+\tau^{-4}\sum_{j=1}^{3}|\nabla(r^j(\Delta(\varphi+i\psi)))|^2 \\
&+\tau^{-4}\sum_{j=1}^{3}|\nabla(\nabla(\varphi+i\psi)\cdot\nabla r^j)|^2+\tau^{-6}\sum_{j=1}^{3}|\nabla(\Delta r^j)|^2\}].
 \end{split}
\end{equation*}
 
\subsection{Estimating the upper bound for $\Vert\nabla u_s\Vert_{L^q(D)}$}
For $1\leq q \leq 2,$
\begin{equation}\label{upper_grad}
\begin{split}
& \int_D \vert \nabla u_s\vert^qdx \leq C \int_D e^{-\tau q \varphi}\sum_{j=1}^{3}[\vert A\vert^q + \vert B\vert^q]
 \leq C \int_De^{-\tau q \varphi}[1 + \tau^{q} + \tau^{-q}+ \tau^{2q} + \tau^{-2q}+\tau^{3q}]  \\
& + \tau^q\int_D e^{-\tau q \varphi}[\tau^{2q}\vert\nabla(\varphi +i\psi)\cdot r\vert^q + \tau^q\vert\nabla\cdot r\vert^q + \sum_{j=1}^{3}\tau^q\vert\frac{\partial}{\partial x_j}(\nabla(\varphi+i\psi)\cdot r)\vert^q] \\
& + \tau^q\int_D\sum_{j=1}^{3} e^{-\tau q \varphi}[\vert\frac{\partial}{\partial x_j}(\nabla\cdot r)\vert^q + \tau^q\vert r^j\Delta(\varphi+i\psi)\vert^q + \tau^q\vert\nabla r^j\vert^q + \vert\Delta r^j\vert^q]\\
& + C \int_D e^{-\tau q \varphi}[\tau^{2q}\vert\nabla(\nabla(\varphi+i\psi)\cdot r)\vert^q + \tau^q\vert\nabla(\nabla\cdot r)\vert^q + \tau^q\sum_{j=1}^{3}\vert\nabla(\frac{\partial}{\partial x_j}(\nabla(\varphi+i\psi)\cdot r))\vert^q] \\
& + C \int_D e^{-\tau q \varphi}\sum_{j=1}^{3}[\nabla(\frac{\partial}{\partial x_j}(\nabla\cdot r)) + \tau^q\vert\nabla(r^j\Delta(\varphi+i\psi))\vert^q + \tau\vert\nabla(\nabla(\varphi+i\psi)\cdot\nabla r^j)\vert^q + \vert\nabla(\Delta r^j)\vert^q].
\end{split}
\end{equation}
By simplifying, we get
\begin{equation}\label{lasteq1}
 \begin{split}
  \int_D|\nabla u_s|^qdx
& \leq c\tau^{3q}\int_De^{-q\tau\varphi}[1+\tau^{-q}+\tau^{-2q}+\tau^{-3q}+\tau^{-4q}+\tau^{-5q}\\
& +c_1\{(1+\tau^{-q})|r|^q+\tau^{-q}|\nabla r|^q+(\tau^{-2q}+\tau^{-3q})|\nabla\nabla r|^q+\tau^{-3q}|\nabla(\Delta r)|^q\}].
 \end{split}
\end{equation}
\end{section}

\end{document}